\documentclass[11pt]{article}
\usepackage{a4wide}
\usepackage[utf8]{inputenc}
\usepackage[a4paper, margin=2.3cm]{geometry}
\usepackage{amsmath,amssymb,amsthm}
\usepackage{thmtools}
\usepackage{thm-restate}
\usepackage{color}
\usepackage{xcolor}
\usepackage{parskip}
\usepackage{hyperref}
\usepackage{tikz-cd}
\usepackage{tikz}
\usetikzlibrary{matrix,arrows.meta}
\usepackage{parskip}
\usepackage{setspace}
\usepackage{graphicx}
\usepackage{enumitem}
\usepackage{comment}
\usepackage{float} 

\newtheorem{theorem}{Theorem}[section]
\newtheorem{remark}[theorem]{Remark}
\newtheorem{corollary}[theorem]{Corollary}

\newtheorem{lemma}[theorem]{Lemma}
\newtheorem{question}[theorem]{Question}
\newtheorem{proposition}[theorem]{Proposition}

\newtheorem{conjecture}[theorem]{Conjecture}

\newtheorem*{definition*}{Definition}

\usepackage{latexsym}
\usepackage{thm-restate}
\usepackage[capitalise]{cleveref}
\crefname{equation}{}{}
\crefname{figure}{{\sc Figure}}{{\sc Figure}}
\crefname{subsection}{Subsection}{Subsections}

\allowdisplaybreaks

\newtheorem*{innerreptheorem}{Theorem}
\newcommand{\reptheoremname}{}

\begin{document}

\title{Intersection patterns and connections to distance problems}

%\author{Alex Iosevich\and Thang Pham\and Nguyen Dac Quan \and Steven Senger\and Boqing Xue}

\author{Thang Pham \and Semin Yoo}
\date{}
\author{Thang Pham\footnote{University of Science, Vietnam National University, Hanoi, Vietnam. Email: thangpham.math@vnu.edu.vn}\and Semin Yoo\footnote{Discrete Mathematics Group, Institute for Basic Science, Daejeon, South Korea. Email: syoo19@ibs.re.kr}}
\maketitle
\begin{abstract} 
Let $A$ and $B$ be sets in $\mathbb{F}_q^d$, where $\mathbb{F}_q$ is a finite field of odd order $q$. In this paper, we study the size of the intersection $A\cap f(B)$, where $f$ runs through a set of transformations. More precisely, we focus on the case where the set of transformations is given by rigid motions. We prove that if $A, B\subset \mathbb{F}_q^d$ satisfy some natural conditions, then, for almost every $g\in O(d)$, there are at least $\gg q^d$ elements $z\in \mathbb{F}_q^d$ such that 
\[|A\cap (g(B)+z)| \sim \frac{|A||B|}{q^d}.\]
This implies that $|A-gB|\gg q^d$ for almost every $g\in O(d)$. In the language of expanding functions, with $|A|\le |B|$, we also show that the image $A-gB$ grows exponentially. In two dimensions, the result simply says that if $|A|=q^x$ and $|B|=q^y$, with $0<x\le y<2$, then for almost every $g\in O(2)$, there exists $\epsilon=\epsilon(x, y)>0$ such that $|A-gB|\gg |B|^{1+\epsilon}$. To prove these results, we need to develop new and robust incidence bounds between points and rigid motions by using a number of techniques including algebraic methods and discrete Fourier analysis. Our results are essentially sharp in odd dimensions. In the prime field plane, we further employ recent \(L^2\) distance bounds and point--line/plane incidence machinery to derive improvements. In particular, notable applications include:
\begin{enumerate}
    \item We prove a strong prime field analogue of a question of Mattila related to the Falconer distance problem: let $A, B\subset \mathbb{F}_p^2$, if \(p\equiv 3\!\!\mod 4\) and \(|A|,|B|\ge p^{5/4}\), then for almost every \(g\in O(2)\) there are \(\gg p^2\) shifts \(z\) with
\[
|A\cap(gB+z)|\sim \frac{|A||B|}{p^2}.
\]
\item We prove the rotational Erd\H{o}s-Falconer distance problem: let $A, B\subset \mathbb{F}_p^2$, if \(p\equiv 3\!\!\mod 4\) and \(|A|,|B|\ge p\), then the distance set \(\Delta(A,gB)\) covers a positive proportion of all elements in the field for almost every \(g\in SO(2)\).
\item We prove a quadratic expansion law: let $A, B\subset \mathbb{F}_p^2$, if \(p\equiv 3\!\!\mod 4\) and \(|A|=|B|=N\ll p^{3/4}\), then for almost every \(g\in O(2)\), one has $
|A-gB|\gg N^2.
$
\end{enumerate}
Taken together, the results in this paper present a robust two-way link between intersection phenomena and distance problems over finite fields, with dimension-uniform consequences and sharpness in several ranges. 
\vspace{0.2cm}


\end{abstract}

\tableofcontents

\section{Introduction}
\subsection{Background and motivation (continuous setting)}
Let $A$ and $B$ be compact sets in $\mathbb{R}^d$. One of the fundamental problems in Geometric Measure Theory is to study the relations between the Hausdorff dimensions of $A$, $B$, and $A\cap f(B)$, where $f$ runs through a set of transformations. 

This study has a long history in the literature. A classical theorem due to Mattila \cite[Theorem 13.11]{M95} or \cite[Theorem 7.4]{M15} states that for Borel sets $A$ and $B$ in $\mathbb{R}^d$ of Hausdorff dimension $s_A$ and $s_B$ with
    \[s_A+s_B>d~\mbox{and}~s_B>\frac{d+1}{2}\]
and assume in addition that the Hausdorff measures satisfy $\mathcal{H}^{s_A}(A)>0$ and $\mathcal{H}^{s_B}(B)>0$, then, for almost every $g\in O(d)$, one has 
    \[\mathcal{L}^d\left(\left\lbrace z\in \mathbb{R}^d\colon \dim_H(A\cap (z-gB))\ge s_A+s_B-d \right\rbrace\right)>0.\]
This result implies that for almost every $g\in O(d)$, the set of $z$s such that $\dim_H(A\cap (z-gB))\ge s_A+s_B-d$ has positive Lebesgue measure. 

Related intersection problems have been studied for other transformation groups and for various projection families (see \cite{BP, DF, EKM, EIT, K1986, M1984, MO, MY}, and the references therein).

\subsection{Finite field formulation and main questions}
Let \(\mathbb{F}_q\) be a finite field of order \(q\), where \(q\) is an odd prime power. 

In this paper, we introduce the finite field analogue of this topic, and study the primary properties with an emphasis on the group of orthogonal matrices. More precisely, we consider the following two main questions. 

\begin{restatable}{question}{ques}
\label{ques}
Let $A,B \subset \mathbb{F}_q^d$ and $g \in O(d)$. 
Under what hypotheses on $A$, $B$, and $g$ does one have
\begin{equation}\label{eq:Q1-lower}
  |A \cap (g(B)+z)| \ge \frac{|A||B|}{q^d}
\end{equation}
for almost every $z \in \mathbb{F}_q^d$, and under what additional hypotheses
can this be strengthened to the asymptotic
\begin{equation}\label{eq:Q1-asymp}
  |A \cap (g(B)+z)| \sim \frac{|A||B|}{q^d}
\end{equation}
for almost every $z \in \mathbb{F}_q^d$?
\end{restatable}

\begin{question}
\label{ques:rigid}
Let $P \subset \mathbb{F}_q^{2d}$ and $g \in O(d)$, and define
\[
  S_g(P) := \{x - g y : (x,y) \in P\} \subset \mathbb{F}_q^d .
\]
Under what hypotheses on $P$ and $g$ can one guarantee that
\[
  |S_g(P)| \gg q^d ?
\]
\end{question}

\noindent\textbf{Notation.}
Throughout the paper, we use the following asymptotic notation.
We write $X \ll Y$, $Y \gg X$, or $X = O(Y)$ to denote that $|X| \le CY$ for some absolute constant $C > 0$.
If we need $C$ to depend on a parameter $\alpha$, we write $X \ll_\alpha Y$ or $X = O_\alpha(Y)$.
We write $X \sim Y$ to mean that $X \ll Y$ and $Y \ll X$.
Implicit constants may depend on the ambient dimension $d$ unless otherwise indicated.
The notation $o(1)$ denotes a quantity that tends to zero as the relevant parameter (usually $q$ or $p$) tends to infinity.
For sets $A,B \subseteq \mathbb{F}_q^d$ and $g \in O(d)$, we define the difference set
\[
A - gB := \{a - gb : a \in A,\ b \in B\},
\]
and the distance set
\[
\Delta(A,B) := \{\|a - b\| : a \in A,\ b \in B\},
\]
where $\|x\| := x_1^2 + \cdots + x_d^2$ denotes the standard quadratic form on $\mathbb{F}_q^d$. Throughout, $p$ denotes an odd prime and $\mathbb{F}_p$ the finite field with $p$ elements.
When we write $\mathbb{F}_q$, we do not assume that $q$ is prime unless explicitly stated.
 By the phrase ``the property $X(g)$ holds for almost every $g\in O(d)$"
(respectively, ``for almost every $g\in SO(d)$"), we mean that there exists
an absolute constant $c>0$, independent of $q$, $d$, and of the sets under consideration, such that $X(g)$ holds for at least $c|O(d)|$
elements of $O(d)$ (respectively, at least $c|SO(d)|$ elements of $SO(d)$).

\subsection{Main results}
The following theorems and applications present our central results in the balanced case ($|A|\sim |B|$). The results are essentially sharp in odd dimensions. A comprehensive treatment of the unbalanced case ($|A|\ll |B|$ or vice versa), together with further discussion and examples of sharpness, will be given in later sections.

The first result is on Intersection pattern I.
\begin{theorem}\label{mot}
Let $A$ and $B$ be sets in $\mathbb{F}_q^d$ with $|A|=|B|=N$. There exists a subset $E\subset O(d)$ such that for any $g\in O(d)\setminus E$, there are at least $\gg q^d$ elements $z$ satisfying $$|A\cap (g(B)+z)| \sim \frac{N^2}{q^d}.$$
In particular, there exist absolute positive constants $C_1$, $C_2$, and $C_3$ with the following properties.
\begin{enumerate}
    \item For all $d\ge 2$, one has $|E|\le C_1\frac{|O(d)|q^{d+1}}{N^2}.$ 
    \item If $d=2$ and $q\equiv 3\mod 4$, then $|E|\le  C_2\frac{q^{3}}{N^{3/2}}$.
    \item If $d=2$, $q\equiv 3\mod 4$ is a prime, and $q^{5/4}\le N\le q^{4/3}$, then $|E|\le C_3\frac{q^{9/4}}{N}.$
\end{enumerate}
\end{theorem}
It follows from \textnormal{(1)}, \textnormal{(2)}, and \textnormal{(3)} that
$|E| < |O(d)|$ provided $N$ is at least a constant multiple of 
$q^{\frac{d+1}{2}}$, $q^{\frac{4}{3}}$, and $q^{\frac{5}{4}}$, respectively.

In the language of expanding functions, we have a quadratic expansion law.
\begin{corollary}\label{cosep7}
Let $q\equiv 3\mod 4$, and $A,B \subset \mathbb{F}_q^{2}$ with $|A|=|B|=N$. 
\begin{enumerate}
\item If $N\ll q^{1/2}$, then for almost every $g\in O(2)$, we have $|A-gB|\gg N^2$. 
    \item If $q^{1/2}\ll N\ll q^{4/3}$, then for almost every $g\in O(2)$, we have $|A-gB|\gg N^{3/2}$. 
    \item If $q$ is a prime and $N\ll q^{3/4}$, then for almost every $g\in O(2)$, we have 
$|A-gB| \gg N^2.$ 
\end{enumerate}
\end{corollary}
The next result is on Intersection pattern II.
\begin{theorem}\label{thmNov17} Let $P$ be a set in $\mathbb{F}_q^{2d}$. There exists $E\subset O(d)$ such that, for any $g\in O(d)\setminus E$, one has $|S_g(P)|\gg q^d$. In particular, there exists an absolute positive constant $C_1$ such that $|E|\le C_1\frac{q^{d+1}|O(d)|}{|P|}$ for all $d\ge 2$. 

Moreover, if $d=2$, $q\equiv 3\mod 4$ is a prime, $P = A\times B$ with $A,B\subset\mathbb{F}_q^2$ and $|A|=|B|\ge q$, then there exists an absolute positive constant $C_2\in (0, 1)$ such that $|E|\le (1-C_2)|O(2)|$.
\end{theorem}
As a consequence, we prove the Rotational Erd\H{o}s-Falconer distance problem in two dimensions.
\begin{corollary}\label{coro1NOV}
Let $q\equiv 3\mod{4}$ be a prime.
    Let $A, B\subset \mathbb{F}_q^2$ with $|A|=|B|\ge q$. Then, for almost every $g\in SO(2)$, one has $|\Delta(A, gB)|\gg q$. 
\end{corollary}

\subsection{Ideas of the proofs and connections to distance problems}
We develop a framework that reduces the problems to an incidence type question, namely, incidences between points and rigid motions. In $\mathbb{R}^2$, such an incidence structure has been studied intensively in the breakthrough solution of the Erd\H{o}s distinct distances problem \cite{E-S, GK}. In the present paper, we go in the \emph{reverse} direction: we derive new incidence theorems from $L^2$ distance estimates over finite fields.
This allows us to import recent progress on distance problems into incidence geometry, and in turn to obtain a fairly complete picture of intersection patterns in all dimensions over finite fields.

More precisely, the present paper provides two types of incidence results: over arbitrary fields $\mathbb{F}_q$ and over prime fields $\mathbb{F}_p$. The key technical challenge involves bounding sums of the form \[\sum_{||m||=||m'||} |\widehat{A}(m)|^2|\widehat{B}(m')|^2\] for sets $A,B \subset \mathbb{F}_q^d$.
While we use results from the Restriction theory due to Chapman, Erdogan, Hart, Iosevich, and Koh \cite{chap}, and Iosevich, Koh, Lee, Pham, and Shen \cite{Zero} to bound this sum effectively over arbitrary finite fields, the proofs of better estimates over prime fields are based on the recent $L^2$ distance estimate due to Murphy, Petridis, Pham, Rudnev, and Stevens \cite{MPPRS} which was proved by using algebraic methods and Rudnev's point-plane incidence bound \cite{rudnev}. Overall, the framework presents surprising applications of the Erd\H{o}s-Falconer distance problem. 

In the continuous setting, Mattila \cite{Mat23} asked whether the condition $\dim_H(A)=\dim_H(B)>\frac{5}{4}$ is sufficient to guarantee that $\mathcal{L}^2(gA-B)>0$, building on techniques of Guth, Iosevich, Ou, and Wang \cite{alex-fal} for the Falconer distance problem.
As a consequence of Theorem \ref{mot} (3), we resolve the prime field version of this question in stronger form: if $|A|, |B|\ge p^{5/4}$, then for almost every $g\in O(2)$, there exist at least $\gg p^2$ elements $z$ such that $$|A\cap (g(B)+z)|\sim \frac{|A||B|}{p^2}.$$ Moreover, if we only need a weaker conclusion that $|A\cap (gB+z)|\ge 1$ for almost every $g\in O(2)$ and for at least $\gg p^2$ elements $z$, then the exponent $\mathbf{\frac{5}{4}}$ can be reduced to $\mathbf{1}$ (Theorem \ref{thmNov17}), and this is optimal. To prove such a result, we used an observation that an incidence between points and rigid motions in $\mathbb{F}_p^2$ can be reduced to an incidence between points and lines in $\mathbb{F}_p^3$ with an appropriate parametrization. Such a parametrization was introduced by Bennett, Iosevich, and Pakianathan \cite{B}. With this connection in hand, the incidence bound between points and lines in $\mathbb{F}_p^3$ due to Ellenberg and Hablicsek \cite{EHH} becomes central to our argument. This intersection result implies directly the Rotational Erd\H{o}s–Falconer problem in the plane (Corollary \ref{coro1NOV}): let $A, B\subset \mathbb{F}_p^2$, if \(p\equiv 3\!\!\mod 4\) and \(|A|,|B|\ge p\), then the distance set \(\Delta(A,gB)\) covers a positive proportion of all elements in the field for almost every \(g\in SO(2)\).

Putting together, this paper establishes a structural interdependence between intersection questions and distance problems over finite fields, providing tools to transfer quantitative estimates in both directions and to obtain strong, dimension-uniform conclusions.

\noindent\textbf{Organization of the paper.} 
The remainder of the paper is organized as follows. 
Section~2 develops the key preliminary material over both arbitrary finite fields and prime fields, including bounds on Fourier sums of the form 
\(\sum_{\|m\|=\|m'\|} |\widehat{A}(m)|^2 |\widehat{B}(m')|^2\), which form the main technical input for the incidence theorems. 
Section~3 establishes incidence bounds between points and rigid motions, providing dimension-uniform estimates together with refined bounds in the planar prime-field case. 
These incidence theorems are then applied in Section~4 to address the intersection pattern problem (Question~1.1), yielding explicit bounds on the size of the exceptional set in \(O(d)\) for which intersections fail to have the expected cardinality. 
Section~5 explores connections with Mattila’s question and the Falconer distance problem, and in particular establishes a rotational Erd\H{o}s--Falconer distance theorem in two dimensions. 
In Section~6, we extend the analysis to general sets via Intersection pattern~II (Question~1.2), characterizing conditions under which \(|S_g(P)| \gg q^d\) for almost every \(g \in O(d)\). 
Section~7 investigates growth estimates for difference sets \(A - gB\) under orthogonal transformations, providing both balanced and unbalanced results for various size regimes. 
We conclude with acknowledgements in Section~8.

\section{Preliminaries: key lemmas}
Let $f\colon \mathbb{F}_q^n\to \mathbb{C}$ be a complex valued function. The Fourier transform $\widehat{f}$ of $f$ is defined by
$$ \widehat{f}(m):=q^{-n} \sum_{x\in \mathbb F_q^n} \chi(-m\cdot x) f(x),$$ 
here, we denote by $\chi$ a nontrivial additive character of $\mathbb{F}_q$. Note that $\chi$ satisfies the following orthogonality property
$$ \sum_{\alpha\in \mathbb F_q^n} \chi(\beta\cdot \alpha)
=\left\{\begin{array}{ll} 0\quad &\mbox{if}\quad \beta\ne (0,\ldots, 0),\\
q^n\quad &\mbox{if}\quad \beta=(0,\ldots,0). \end{array}\right.$$
We also have the Fourier inversion formula as follows
$$ f(x)=\sum_{m\in \mathbb F_q^n} \chi(m\cdot x) \widehat{f}(m).$$ 
With these notations in hand, the Plancherel theorem states that
$$ \sum_{m\in \mathbb F_q^n} |\widehat{f}(m)|^2 =q^{-n}\sum_{x\in \mathbb F_q^n} |f(x)|^2.$$
In this paper, we denote the quadratic character of $\mathbb{F}_q$ by $\eta$, precisely, for $s\ne 0$, $\eta(s)=1$ if $s$ is a square and $-1$ otherwise. The convention that $\eta(0)=0$ will be also used in this paper.

This section is devoted to proving upper bounds of the following sum 
\[\sum_{||m||=||m'||}|\widehat{A}(m)|^2|\widehat{B}(m')|^2 ~\text{ for any }A, B\subset \mathbb{F}_q^d,\]
which is the key step in the proofs of incidence theorems. 

\subsection{Results over arbitrary finite fields}
We first start with a direct application of the Plancherel theorem. 
\begin{theorem}\label{plan-ge}
    Let $A, B$ be sets in $\mathbb{F}_q^d$. Then we have 
    \[\sum_{||m||=||m'||}|\widehat{A}(m)|^2|\widehat{B}(m')|^2\le \frac{|A||B|}{q^{2d}}.\]
\end{theorem}
To improve this result, we need to recall a number of lemmas in the literature. 

For any $j\ne 0$, let $S_j$ be the sphere centered at the origin of radius $j$ defined as follows: 
$$S_j:=\{x\in \mathbb{F}_q^d\colon x_1^2+\cdots+x_d^2=j\}.$$
The next lemma provides the precise form of the Fourier decay of $S_j$ for any $j\in \mathbb{F}_q$. 
A proof can be found in \cite{TAMS} or \cite{Kohsun}.
\begin{lemma}\label{large-fourier}
For any $j\in \mathbb{F}_q$, we have 
\[ \widehat{S_j}(m) = q^{-1} \delta_0(m) + q^{-d-1}\eta^d(-1) G_1^d(\eta, \chi) \sum_{r \in {\mathbb F}_q^*} 
\eta^d(r)\chi\Big(jr+ \frac{\|m\|}{4r}\Big),\]
where $\eta$ is the quadratic character, and $\delta_0(m)=1$ if $m=(0,\ldots, 0)$ and $\delta_0(m)=0$ otherwise.

Moreover, for $m, m'\in \mathbb{F}_q^d$, we have  
 \[\sum_{j\in \mathbb{F}_q}\widehat{S_j}(m)\overline{\widehat{S_j}(m')}=\frac{\delta_0(m)\delta_0(m')}{q}+\frac{1}{q^{d+1}}\sum_{s\in \mathbb{F}_q^*}\chi(s(||m||-||m'||)).\]
\end{lemma}
For $A\subset \mathbb{F}_q^d$, define 
\[M^*(A)=\max_{j\ne 0} \sum_{m\in S_j}|\widehat{A}(m)|^2, ~\mbox{and}~M(A)=\max_{j\in \mathbb{F}_q} \sum_{m\in S_j}|\widehat{A}(m)|^2.\]
We recall the following result from \cite{Kohsun}, which is known as the finite field analogue of the spherical average in the
classical Falconer distance problem \cite[Chapter 3]{M15}. 
\begin{theorem}\label{m(A)}
Let $A\subset \mathbb{F}_q^d$. We have 
\begin{enumerate}
    \item If $d=2$, then $M^*(A)\ll q^{-3}|A|^{3/2}$. 
    \item If $d\ge 4$ even, then $M^*(A)\ll \min \left\lbrace \frac{|A|}{q^{d}},~ \frac{|A|}{q^{d+1}}+\frac{|A|^2}{q^{\frac{3d+1}{2}}}\right\rbrace$.
     \item If $d\ge 3$ odd, then $M(A)\ll  \min \left\lbrace \frac{|A|}{q^{d}},~ \frac{|A|}{q^{d+1}}+\frac{|A|^2}{q^{\frac{3d+1}{2}}}\right\rbrace$.
\end{enumerate}
\end{theorem}
In some specific dimensions, a stronger estimate was proved in \cite{Zero} for the sphere of zero radius.
\begin{theorem}\label{zero-sphere}
    Let $A\subset \mathbb{F}_q^d$. Assume $d\equiv 2\mod{4}$ and $q\equiv 3\mod{4}$, then we have 
    \[\sum_{m\in S_0}|\widehat{A}(m)|^2\ll  \frac{|A|}{q^{d+1}}+\frac{|A|^2}{q^{\frac{3d+2}{2}}}.\]
\end{theorem}
We are now ready to improve \cref{plan-ge}.
\begin{theorem}\label{plan-ge-1}
 Let $A, B$ be sets in $\mathbb{F}_q^d$. Assume that either ($d\ge 3$ odd) or ($d\equiv 2\mod 4$ and $q\equiv 3\mod 4$), then the following hold. 
 \begin{enumerate}
     \item If $|A|\le q^{\frac{d-1}{2}}$, then
     \[\sum_{||m||=||m'||}|\widehat{A}(m)|^2|\widehat{B}(m')|^2\le \frac{|A||B|}{q^{2d+1}}.\]
     \item If $q^{\frac{d-1}{2}}\le |A|\le q^{\frac{d+1}{2}}$, then 
     \[\sum_{||m||=||m'||}|\widehat{A}(m)|^2|\widehat{B}(m')|^2\le \frac{|A|^2|B|}{q^{\frac{5d+1}{2}}}.\]
 \end{enumerate}
\end{theorem}
\begin{proof}
The proof follows directly from \cref{m(A)} and \cref{zero-sphere}. More precisely,
    \[\sum_{||m||=||m'||}|\widehat{A}(m)|^2|\widehat{B}(m')|^2\le \max_{t\in \mathbb{F}_q}\sum_{m\in S_t}|\widehat{A}(m)|^2\cdot \sum_{m'\in \mathbb{F}_q^d}|\widehat{B}(m')|^2\le \max_{t\in \mathbb{F}_q}\sum_{m\in S_t}|\widehat{A}(m)|^2\cdot \frac{|B|}{q^d}.\]
This completes the proof.
\end{proof}

In two dimensions, we can obtain a better estimate as follows. 
\begin{theorem}\label{key-2-d}
    Let $A, B$ be sets in $\mathbb{F}_q^2$. Assume in addition that $q\equiv 3\mod 4$, then we have 
    \[\sum_{||m||=||m'||}|\widehat{A}(m)|^2|\widehat{B}(m')|^2\ll \frac{|A||B|}{q^5}\cdot\min\{|A|^{1/2}, |B|^{1/2}\}.\]
\end{theorem}
\begin{proof}
    We note that when $q\equiv 3\mod{4}$, the circle of radius zero contains only one point which is $(0, 0)$, so
\[|\widehat{A}(0, 0)|^2|\widehat{B}(0, 0)|^2=\frac{|A|^2|B|^2}{q^8}.\]
Applying \cref{m(A)}, as above, one has
\[\sum_{||m||=||m'||\ne 0}|\widehat{A}(m)|^2|\widehat{B}(m')|^2\le M^*(A)\sum_{m}|\widehat{B}(m)|^2\ll \frac{|A|^{3/2}|B|}{q^5} ,\]
and 
\[\sum_{||m||=||m'||\ne 0}|\widehat{A}(m)|^2|\widehat{B}(m')|^2\le M^*(B)\sum_{m}|\widehat{A}(m)|^2\le \frac{|A||B|^{3/2}}{q^5}.\]
Thus, the theorem follows.
\end{proof}

\subsection{Results over prime fields}
To improve \cref{key-2-d} over prime fields, we need to introduce the following notation. 

For $P\subset \mathbb{F}_q^{2d}$, 
define 
\[N(P):=\#\{(x, y, u, v)\in P\times P\colon ||x-u||=||y-v||\}.\]
The following picture describes the case $P=A\times B$.
\begin{center}
    \begin{figure}[h!]
\includegraphics[width=0.6\textwidth]{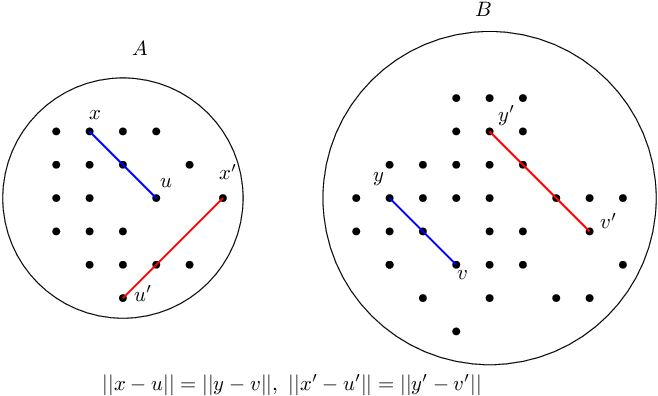}
\label{vcdim3figure}
\end{figure}
\end{center}
We note that $N(A\times B)$ counts the number of pairs in $A\times A$ and $B\times B$ of the same distance. This quantity is not the same as  the $L^2$ distance estimate of the set $A\times B$. 

While the sum $\sum_{||m||=||m'||}|\widehat{A}(m)|^2|\widehat{B}(m')|^2$ can be bounded directly over arbitrary finite fields, the strategy over prime fields is different. More precisely, we will use a double counting argument to bound $N(P)$. The first bound is proved in the next theorem which presents a connection between the sum and the magnitude of $N(P)$. The second bound (\cref{small-sets}, \cref{cor-smallsets}) for $N(P)$ is due to Murphy, Petridis, Pham, Rudnev, and Stevens \cite{MPPRS} which was proved by using algebraic methods and Rudnev's point-plane incidence bound \cite{rudnev}.
\begin{theorem}\label{framework}
    Let $P=A \times B$ for $A,B \subset \mathbb{F}_{q}^{d}$. Then we have 
 \[N(P)=\frac{|P|^2}{q}-q^{3d-1}\sum_{||m||\ne ||m'||}|\widehat{A}(m)|^2|\widehat{B}(m')|^2+q^{3d}\sum_{||m||=||m'||}|\widehat{A}(m)|^2|\widehat{B}(m')|^2.\]
\end{theorem}
\begin{proof}
For any $j\in \mathbb{F}_q$ and $A\subset \mathbb{F}_q^d$, we define \[\nu_A(j):=\# \{(x, y)\in A\times A\colon ||x-y||=j\}.\]
Therefore, for any $j\ne 0$, we have
\begin{align*}
\nu_A(j)&=\sum_{x, y\in \mathbb{F}_q^d}A(x)A(y)S_j(x-y)=\sum_{x, y\in \mathbb{F}_q^d} A(x)A(y)\sum_{m\in \mathbb{F}_q^d} \widehat{S_j}(m)\chi(m(x-y))\\
&=q^{2d}\sum_{m\in \mathbb{F}_q^d}\widehat{S_j}(m)|\widehat{A}(m)|^2.
\end{align*}
Thus,
\[N(P)=\sum_{t\in \mathbb{F}_q}\nu_A(t)\nu_B(t)=q^{4d}\sum_{m, m'}|\widehat{A}(m)|^2|\widehat{B}(m')|^2\sum_{t \in \mathbb{F}_{q}}\widehat{S_t}(m)\overline{\widehat{S_t}}(m').\]
\cref{large-fourier} tells us that 
\[\sum_{t\in \mathbb{F}_q}\widehat{S_t}(m)\overline{\widehat{S_t}}(m')=\frac{\delta_0(m)\delta_0(m')}{q}+\frac{1}{q^{d+1}}\sum_{s\in \mathbb{F}_q^*}\chi(s(||m||-||m'||)).\]
Therefore, 
\begin{align*}
    N(P)&=\frac{|A|^2|B|^2}{q}+q^{3d-1}\sum_{m, m'}|\widehat{A}(m)|^2|\widehat{B}(m')|^2\sum_{s\ne 0}\chi(s(||m||-||m'||)) \\
    &=\frac{|A|^2|B|^2}{q}-q^{3d-1}\sum_{||m||\ne ||m'||}|\widehat{A}(m)|^2|\widehat{B}(m')|^2+q^{3d}\sum_{||m||=||m'||}|\widehat{A}(m)|^2|\widehat{B}(m')|^2.
\end{align*}   
This completes the proof.
\end{proof}
\begin{theorem}\label{small-sets}
For $A\subset \mathbb{F}_p^2$, $p\equiv 3\mod{4}$ with $|A|\ll p^{4/3}$, and $P=A\times A$, we have 
\[N(P)=\#\{(x, y, z, w)\in A^4\colon ||x-y||=||z-w||\}\le \frac{|A|^4}{p}+C\min \left\lbrace p^{2/3}|A|^{8/3}+p^{1/4}|A|^3,  |A|^{10/3}\right\rbrace,\]
for some large constant $C>0$.
\end{theorem}

\begin{corollary}\label{cor-smallsets}
For $A\subset \mathbb{F}_p^2$, $p\equiv 3\mod{4}$ with $|A|\ll p^{4/3}$, and $P=A\times A$, there exists a large constant $C$ such that the following hold. 
\begin{enumerate}
\item If $|A|\le p$, then $N(P)\le  C|A|^{10/3}$. 
\item If $p\le |A|\le p^{5/4}$, then $N(P)\le p^{-1}|A|^4+Cp^{2/3}|A|^{8/3}$.
\item If $p^{5/4}\le |A|\le p^{3/2}$, then $N(P)\le p^{-1}|A|^4+Cp^{1/4}|A|^3$. 
\end{enumerate}
\end{corollary}
\begin{remark}
We remark that  \cite[Theorem 4]{MPPRS} presents a bound on the number of isosceles triangles in $A$. This implies the bound for $N(P)$ as stated in \cref{small-sets} since $N(P)$ is at most the number of isosceles triangles times the size of $A$ by the Cauchy-Schwarz inequality. 
\end{remark}

\begin{theorem}\label{th:novelty}
    Let $A, B\subset \mathbb{F}_p^2$ with $|A|\le |B|$ and $p\equiv 3\mod 4$. Define 
    \[I_{A, B}:=p^6\sum_{||m||=||m'||}|\widehat{A}(m)|^2|\widehat{B}(m')|^2.\]
Then the following hold. 
\begin{enumerate}
    \item If $|A| \le p$ and $p^{5/4}\le |B|\le p^{4/3}$, then $I_{A, B}\ll p^{1/8}(p|A|^2+|A|^{10/3})^{1/2}|B|^{3/2}$.

\item If $|A|\le p$ and $p\le |B|\le p^{5/4}$, then $I_{A, B}\ll (p|A|^2+|A|^{10/3})^{1/2}\cdot p^{1/3}|B|^{4/3}$.

\item  If $|A|\le p$ and $|B|\le p$, then $I_{A, B}\le (p|A|^2+|A|^{10/3})^{1/2}\cdot (p|B|^2+|B|^{10/3})^{1/2}$.

\item If $p \le |A| \le p^{5/4}$ and $p \le |B| \le p^{5/4}$, then $I_{A, B}\ll p^{2/3}|A|^{4/3}|B|^{4/3}$.

\item If $p \le |A| \le p^{5/4}$ and $p^{5/4}\le |B|\le p^{4/3}$, then $I_{A, B}\ll p^{11/24}|A|^{4/3}|B|^{3/2}$.
\item  If $p^{5/4}\le |A|\le p^{4/3}$ and $p^{5/4}\le |B|\le p^{4/3}$, then $I_{A, B}\ll p^{1/4}|A|^{3/2}|B|^{3/2}.$ 
\item  If $|A|\le p$ and $|B|>p^{4/3}$, then $I_{A, B}\ll p^{1/2}(p|A|^2+|A|^{10/3})^{1/2}|B|^{5/4}$. 

\item  If $p\le |A|\le p^{5/4}$ and $|B|>p^{4/3}$, then $I_{A, B}\ll p^{5/6}|A|^{4/3}|B|^{5/4}$. 

\item  If $p^{5/4}\le |A|\le p^{4/3}$ and $|B|>p^{4/3}$, then $I_{A, B}\ll p^{5/8}|A|^{3/2}|B|^{5/4}$.
\end{enumerate}
\end{theorem}
To see how good this theorem is, we need to compare with the results obtained by \cref{plan-ge-1} and \cref{key-2-d}. More precisely, the two theorems give 
\[I_{A, B}\ll \begin{cases}
    p|A|^{3/2}|B|&~~\mbox{if}~|A|\ge p\\
    p^{1/2}|A|^2|B|&~~\mbox{if}~p^{\frac{1}{2}}\le |A|<p\\
    p|A||B|&~~\mbox{if}~|A|<p^{\frac{1}{2}}
\end{cases}.\]
The following table gives the information we need. 
\vspace{2mm}
\begin{table}[!ht]
    \centering
   	\begin{tabular}{ |c|c|c|c|c|c|c|} 
 \hline
  & $|A|\le p^{\frac{1}{2}}$ & $p^{\frac{1}{2}}<|A|\le p^{\frac{3}{4}}$&$p^{\frac{3}{4}}< |A|\le p$& $p<|A|\le p^{\frac{5}{4}}$&$p^{\frac{5}{4}}< |A|\le p^{\frac{4}{3}}$ \\ 
 \hline
 $|B|\le p^{\frac{3}{4}}$ & $=$ & $\surd$ & $\diagup $ &$\diagup $ &$\diagup $\\
 \hline 
 $p^{\frac{3}{4}}<|B|\le p$ & $\varnothing$ & $|A|^3>|B|^2$ & $\surd$& $\diagup $ &$\diagup $\\
 \hline
 $p<|B|\le p^{\frac{5}{4}}$ & $\varnothing$ & $|A|>(p|B|)^{\frac{1}{3}}$ & $\surd$& $\surd$&$\diagup $\\
 \hline
 $p^{\frac{5}{4}}<|B|\le p^{\frac{4}{3}}$ & $\varnothing$ & $\varnothing$& $|B|<p^{\frac{3}{4}}|A|^{\frac{2}{3}}$&$\surd$&$\surd$\\
 \hline
 $p^{\frac{4}{3}}<|B|$ & $\varnothing$ & $\varnothing$ & $\varnothing$&$|B|^3<p^2|A|^2$&$|B|<p^{\frac{3}{2}}$\\
 \hline 
\end{tabular}
\vspace{5mm}
    \caption{In this table, by $``="$ we mean the same result, by $``\surd"$ we mean better result, by $``\varnothing"$ we mean weaker result, by $``f(|A|, |B|)"$ we mean better result under the condition $f(|A|, |B|)$, and by $``\diagup"$ we mean invalid range corresponding to $|B|\le |A|$.}
    \label{tab:my_label777}
\end{table}
\begin{proof}
    We have 
\begin{align*}
I_{A, B}&=p^6\sum_{t\in \mathbb{F}_p}\sum_{m, m'\in S_t}|\widehat{A}(m)|^2|\widehat{B}(m')|^2=p^6\sum_{t\in \mathbb{F}_p}\left(\sum_{m\in S_t}|\widehat{A}(m)|^2\right)\cdot\left(\sum_{m'\in S_t}|\widehat{B}(m')|^2\right) \\
&\le p^6\left(\sum_{t\in \mathbb{F}_p}\sum_{m, m'\in S_t}|\widehat{A}(m)|^2|\widehat{A}(m')|^2\right)^{1/2}\cdot \left( \sum_{t\in \mathbb{F}_p}\sum_{m, m'\in S_t}|\widehat{B}(m)|^2|\widehat{B}(m')|^2\right)^{1/2}\\
&=I_{A, A}^{1/2}\cdot I_{B, B}^{1/2}.
\end{align*}
From \cref{framework}, one has 
\[I_{A, A}=N(A\times A)+p^5\sum_{||m||\ne ||m'||}|\widehat{A}(m)|^2|\widehat{A}(m')|^2-\frac{|A|^4}{p},\] 
and 
\[I_{B, B}=N(B\times B)+p^5\sum_{||m||\ne ||m'||}|\widehat{B}(m)|^2|\widehat{B}(m')|^2-\frac{|B|^4}{p}.\]
By the Plancherel, we know that 
\[p^5\sum_{||m||\ne ||m'||}|\widehat{A}(m)|^2|\widehat{A}(m')|^2\le p|A|^2, ~p^5\sum_{||m||\ne ||m'||}|\widehat{B}(m)|^2|\widehat{B}(m')|^2\le p|B|^2.\]
Hence,
\begin{enumerate}
\item If $|A|\le p$, then \cref{cor-smallsets} gives $I_{A, A}\ll p|A|^2+|A|^{10/3}.$
\item  If $p\le |A|\le p^{5/4}$, then \cref{cor-smallsets} gives $I_{A, A}\ll p|A|^2+p^{2/3}|A|^{8/3}\ll p^{2/3}|A|^{8/3}.$
\item  If $p^{5/4}\le |A|\le p^{4/3}$, then \cref{cor-smallsets} gives 
$I_{A, A}\ll p|A|^2+p^{1/4}|A|^{3}\ll p^{1/4}|A|^3$.
\end{enumerate}
The sum $I_{B, B}$ is estimated in the same way, namely, 
\begin{enumerate}
\item If $|B|\le p$, then \cref{cor-smallsets} gives $I_{B, B}\ll p|B|^2+|B|^{10/3}.$
\item  If $p\le |B|\le p^{5/4}$, then \cref{cor-smallsets} gives $I_{B, B}\ll p|B|^2+p^{2/3}|B|^{8/3}\ll p^{2/3}|B|^{8/3}.$
\item  If $p^{5/4}\le |B|\le p^{4/3}$, then \cref{cor-smallsets} gives 
$I_{B, B}\ll p|B|^2+p^{1/4}|B|^{3}\ll p^{1/4}|B|^3$.
\end{enumerate}
When $|B|>p^{4/3}$, we use \cref{key-2-d} to get that $I_{B, B}\ll p|B|^{5/2}$.

Combining these estimates gives us the desired result.
\end{proof}

\subsection{An extension for general sets}
In this subsection, we aim to bound the sum 
\begin{align}\label{eq: sum}
\sum_{\substack{(m, m')\ne (0,0),\\ ||m||=||m'||}}|\widehat{P}(m, m')|^2,  
\end{align}
where $P$ is a general set in $\mathbb{F}_q^{2d}$. 
\begin{theorem}\label{main-this-section}
    Let $P\subset \mathbb{F}_q^d\times \mathbb{F}_q^d$. We have 
    \[q^{3d-1}(q-1)\sum_{\substack{(m, m')\ne (0,0),\\ ||m||=||m'||}}|\widehat{P}(m, m')|^2\ll q^d|P|.\]
\end{theorem}
To prove this result, as in the prime field case, we use a double counting argument to bound $N(P)$. To prove a connection between $N(P)$ and the sum \cref{eq: sum}, a number of results on exponential sums are needed. 

For each $a\in \mathbb{F}_q\setminus \{0\}$,  the Gauss sum $\mathcal{G}_a$ is defined by
$$ \mathcal{G}_a=\sum_{t\in  \mathbb{F}_q\setminus \{0\}} \eta(t) \chi(at).$$
The next lemma presents the explicit form of the Gauss sum, which can be found in \cite[Theorem 5.15]{LN97}.
\begin{lemma}\label{ExplicitGauss}
Let $\mathbb{F}_q$ be a finite field of order $q=p^{\ell}$, where $p$ is an odd prime and $\ell \in {\mathbb N}.$
We have
$$\mathcal{G}_1=\left\{\begin{array}{ll}  {(-1)}^{\ell-1} q^{\frac{1}{2}} \quad &\mbox{if} \quad p \equiv 1 \mod 4 \\
                  {(-1)}^{\ell-1} i^\ell q^{\frac{1}{2}}  \quad &\mbox{if} \quad p\equiv 3 \mod 4.\end{array}\right.$$
\end{lemma}
We also need the following simple lemma; its proof can be found in \cite{KLM}.
\begin{lemma}\label{complete}
For $\beta \in  \mathbb{F}_q^k$ and $s\in \mathbb{F}_q\setminus\{0\}$, we have
$$ \sum_{\alpha \in \mathbb{F}_q^k} \chi( s \alpha \cdot \alpha + \beta \cdot \alpha ) 
=  \eta^k(s) \mathcal{G}_1^k\chi\left( \frac{\|\beta\|}{-4s}\right).$$
\end{lemma}
Let $V\subset \mathbb{F}_q^{2d}$ be the variety defined by 
\[x_1^2+\cdots+x_d^2-y_1^2-\cdots-y_{d}^2=0.\]
The Fourier transform of $V$ can be computed explicitly in the following lemma. 
\begin{lemma}\label{Fourier-V}
    Let $(m, m')\in \mathbb{F}_q^{2d}$.
    \begin{enumerate}
        \item If $(m, m')=(0, 0)$, then 
        \[\widehat{V}(m, m')=\frac{1}{q}+\frac{q^d(q-1)}{q^{2d+1}}.\]
        \item If $(m, m')\ne (0, 0)$ and $||m||=||m'||$, then 
        \[\widehat{V}(m, m')=\frac{q^d(q-1)}{q^{2d+1}}.\]
        \item If $(m, m')\ne (0, 0)$ and $||m||\ne ||m'||$, then \[\widehat{V}(m, m')=\frac{-1}{q^{d+1}}.\]
    \end{enumerate}
\end{lemma}
\begin{proof}
By \cref{complete}, we have 
\begin{align*}
    &\widehat{V}(m, m')=\frac{1}{q^{2d}}\sum_{x,y\in\mathbb{F}_q^d}V(x, y)\chi(-x\cdot m-y\cdot m')\\
    &=\frac{1}{q^{2d+1}}\sum_{x,y\in \mathbb{F}_q^d}\sum_{s\in \mathbb{F}_q}\chi(s(x_1^2+\cdots+x_d^2-y_1^2-\cdots-y_d^2))\chi(-x_1m_1-\cdots-x_dm_d)\chi(-y_1m_1'-\cdots-y_dm_d')\\
    &=\frac{1}{q^{2d+1}}\sum_{x, y}\chi(-x\cdot m)\chi(-y\cdot m')+\frac{1}{q^{2d+1}}\mathcal{G}_1^{2d}\eta^d(-1)\sum_{s\ne 0}\chi\left(\frac{1}{4s}(||m||-||m'||)\right).
\end{align*}
By \cref{ExplicitGauss}, we have $\mathcal{G}_1^{2d}\eta^d(-1)=q^d$. Thus, the lemma follows from the orthogonality of the character $\chi$. 
\end{proof}
In the following, we compute $N(P)$ explicitly which is helpful to estimate the sum \cref{eq: sum}.
\begin{lemma}\label{p-main}
    For $P\subset \mathbb{F}_q^d\times \mathbb{F}_q^d$, we have 
    \[N(P)=\left(\frac{1}{q}+\frac{q-1}{q^{d+1}}\right)|P|^2+q^{3d-1}(q-1)\sum_{\substack{(m, m')\ne (0, 0),\\ ||m||=||m'||}}|\widehat{P}(m, m')|^2-q^{3d-1}\sum_{||m||\ne ||m'||}|\widehat{P}(m,m')|^2.\]
\end{lemma}
\begin{proof}
    We have 
    \begin{align*}
&N(P)=\sum_{x, u, y, v}P(x, u)P(y, v)V(x-y, u-v)\\
&=\sum_{x, u, y, v}P(x, u)P(y, v)\sum_{m, m'}\widehat{V}(m, m')\chi((x-y)m+(u-v)m')\\
&=q^{4d}\sum_{m, m'}\widehat{V}(m,m')|\widehat{P}(m, m')|^2.
    \end{align*}
By \cref{Fourier-V}, we obtain
\[N(P)=\left(\frac{1}{q}+\frac{q-1}{q^{d+1}}\right)|P|^2+q^{3d-1}(q-1)\sum_{\substack{(m, m')\ne (0, 0),\\ ||m||=||m'||}}|\widehat{P}(m, m')|^2-q^{3d-1}\sum_{||m||\ne ||m'||}|\widehat{P}(m,m')|^2,\]
and the lemma follows.
\end{proof}
We now bound $N(P)$ by a different argument. 
\begin{theorem}\label{quadruple}
    For $P\subset \mathbb{F}_q^d\times \mathbb{F}_q^d$. We have 
    \[\left\vert N(P)- \frac{|P|^2}{q}\right\vert \ll q^d|P|.\]
\end{theorem}
\begin{remark}
This theorem is sharp in odd dimensions. More precisely, it cannot be improved to the form 
\[ N(P)\ll \frac{|P|^2}{q} + q^{d-\epsilon}|P|,\]
for any $\epsilon>0$. Since otherwise, it would say that any set $A\subset \mathbb{F}_q^d$ with $|A|\gg q^{\frac{d+1}{2}-\frac{\epsilon}{2}}$ has at least $\gg q$ distances. This is not possible due to examples in \cite{TAMS, IKR}.
\end{remark}
\begin{proof}
    To prove this theorem, we start with the following observation that
    \[||x-u||=||y-v||\]
    can be written as 
    \[-2x\cdot u+2y\cdot v=||y||+||v||-||x||-||u||.\]
We now write $N(P)$ as follows
\begin{align*}
    N(P)&=\frac{1}{q}\sum_{s\in \mathbb{F}_q}\sum_{\substack{(x, y)\in P, \\(u, v)\in P}}\chi\left(s(-2x\cdot u+2y\cdot v-||y||-||v||+||x||+||u||)\right)\\
    &=\frac{|P|^2}{q}+\frac{1}{q}\sum_{s\ne 0}\sum_{\substack{(x, y)\in P, \\(u, v)\in P}}\chi\left(s(-2x\cdot u+2y\cdot v-||y||-||v||+||x||+||u||)\right)\\
    &=\frac{|P|^2}{q}+\frac{1}{q}\sum_{s\ne 0}\sum_{\substack{(x,y, ||x||-||y||)\in P'\\ (u,v) \in P}}\chi(s((x, y)\cdot (-2u, 2v)-||y||-||v||+||x||+||u||))\\
    &=\frac{|P|^2}{q}+\mathtt{Error},
\end{align*}
here $P':=\{(x, y, ||x||-||y||)\colon (x, y)\in P\}\subset \mathbb{F}_q^{2d+1}.$
We now estimate the term $\mathtt{Error}$. 
\begin{align*}
    &\mathtt{Error}^2\le \frac{|P|}{q^2}\sum_{(x, y, t)\in \mathbb{F}_q^{2d+1}}\sum_{s, s'\ne 0}\sum_{\substack{(u, v) \in P,\\ (u', v')\in P}}\chi\left(s(-2x\cdot u+2y\cdot v-||y||-||v||+||x||+||u||)\right)\\&\cdot \chi\left(s'(2x\cdot u'-2y\cdot v'+||y||+||v'||-||x||-||u'||)\right)\\
    &= \frac{|P|}{q^2}\sum_{(x, y, t)\in \mathbb{F}_q^{2d+1}}\sum_{s, s'\ne 0}\sum_{\substack{(u, v) \in P,\\ (u', v')\in P}}\chi(x\cdot (-2su+2s'u'))\cdot \chi(y\cdot (2sv-2s'v'))\cdot \chi(t(s-s'))\\ 
    &\cdot \chi(s(||u||-||v||)-s'(||u'||-||v'||))\\
    &\le |P|^2q^{2d}.
\end{align*}
In other words, we obtain 
\[N(P)\le \frac{|P|^2}{q}+q^d|P|.\]
This completes the proof.
\end{proof}
With \cref{p-main} and \cref{quadruple} in hand, we are ready to prove \cref{main-this-section}.
\begin{proof}[Proof of \textup{\cref{main-this-section}}]
Indeed, one has 
\begin{align*}
    &q^{3d-1}(q-1)\sum_{\substack{(m, m')\ne (0,0),\\ ||m||=||m'||}}|\widehat{P}(m, m')|^2=N(P)-\frac{|P|^2}{q}-\frac{q-1}{q^{d+1}}|P|^2+q^{3d-1}\sum_{||m||\ne ||m'||}|\widehat{P}(m, m')|^2.
\end{align*}
By Plancherel theorem, we have 
\[\sum_{||m||\ne ||m'||}|\widehat{P}(m, m')|^2\le \frac{|P|}{q^{2d}}.\]
So, the theorem follows directly from \cref{quadruple}.
\end{proof}

\section{Incidences between points and rigid motions}\label{subsec: incidence}
This section is devoted to introducing incidence questions and main results, which are key ingredients in the proofs of the main theorems of the paper.

Let $P$ be a set of points in $\mathbb{F}_q^d\times \mathbb{F}_q^d$ and $R$ be a set of rigid motions in $\mathbb{F}_q^d$, i.e. maps of the form $gx+z$ with $g\in O(d)$ and $z\in \mathbb{F}_q^d$. We define the incidence $I(P, R)$ as follows:
\[I(P, R)=\#\{(x, y, g, z)\in P\times R\colon x=gy+z\}.\]
We first provide a universal incidence bound.
\begin{theorem}\label{thm: incidence inequality}
Let $P\subset \mathbb{F}_q^d\times \mathbb{F}_q^d$ and let $R$ be a set of rigid motions in $\mathbb{F}_q^d$. Then we have 
    \[\left\vert I(P, R)-\frac{|P||R|}{q^d}\right\vert\ll q^{(d^{2}-d+2)/4}\sqrt{|P||R|}.\]
\end{theorem}
In this theorem and the next ones, the quantities $|P||R|/q^d$ and $q^{(d^{2}-d+2)/4}\sqrt{|P||R|}$ are referred to as the main and error terms, respectively. 

Under some additional conditions on $d$ and $q$, if one set is of small size compared to the other, then we can prove stronger incidence bounds.
\begin{theorem}\label{thm: incidence1}
     Let $P=A\times B$ for $A, B\subset \mathbb{F}_q^d$ and let $R$ be a set of rigid motions in $\mathbb{F}_q^d$. Assume in addition that either ($d\ge 3$ odd) or ($d\equiv 2\mod 4$ and $q\equiv 3\mod 4$). 
\begin{enumerate}
    \item[\textup{(1)}] If $|A|<q^{\frac{d-1}{2}}$, then 
\[\left\vert I(P, R)-\frac{|P||R|}{q^d}\right\vert\ll q^{(d^{2}-d)/4}\sqrt{|P||R|}.\]
    \item[\textup{(2)}] If $q^{\frac{d-1}{2}}\le |A|\le q^{\frac{d+1}{2}}$, then 
\[\left\vert I(P, R)-\frac{|P||R|}{q^d}\right\vert\ll q^{(d^{2}-2d+1)/4}\sqrt{|P||R||A|}.\]
\end{enumerate}
\end{theorem}
Theorem \ref{thm: incidence inequality} and Theorem \ref{thm: incidence1} are sharp in odd dimensions (see Subsection \ref{subsection-inci-shapr}). 

The next theorems present improvements in two dimensions.
\begin{theorem}\label{thm: incidences2}
Assume that $q\equiv 3\mod 4$. Let $P=A \times B$ for $A,B \subset \mathbb{F}_{q}^{2}$ and let $R$ be a set of rigid motions in $\mathbb{F}_q^2$.
Then we have
\[\left\vert I(P, R)-\frac{|P||R|}{q^2}\right\vert\ll q^{1/2}|P|^{1/2}|R|^{1/2}\min\big(|A|^{1/4},|B|^{1/4}\big).\]
\end{theorem}
A direct computation shows that this incidence result is better than Theorem \ref{thm: incidence inequality} for all $A, B\subset \mathbb{F}_q^2$, and is better than Theorem \ref{thm: incidence1} in the range $|A|\ge q$. 

In a recent paper \cite{pham}, the first author used this theorem to derive results on the distribution of pinned simplices and related questions. As pointed out in \cite{pham}, from the perspective of applications, Theorem \ref{thm: incidences2} is sharp in the sense that the upper bound 
\[
q^{1/2}\, |P|^{1/2}|R|^{1/2}\min\big(|A|^{1/4},|B|^{1/4}\big)
\]
cannot be improved to 
\[
q^{1/2-\epsilon}\, |P|^{1/2}|R|^{1/2}\min\big(|A|^{1/4},|B|^{1/4}\big),
\]
for any $\epsilon > 0$.

In the plane over prime fields, we provide further improvements corresponding to three cases: $|A|\le p ~\mbox{and}~|B|\le p^{4/3}$, $|A|\ge p ~\mbox{and}~|B|\le p^{4/3}$, and $|B|>p^{4/3}$, respectively.

\begin{theorem}[$|A|\le p ~\mbox{\bf and}~|B|\le p^{4/3}$]\label{thm: incidences31}
Assume that $p \equiv 3 \mod{4}$.
Let $P=A\times B$ for $A, B\subset \mathbb{F}_p^2$ with $|A|\le |B|$ and let $R$ be a set of rigid motions in $\mathbb{F}_p^2$. The following hold. 
\begin{enumerate}
\item[\textup{(1)}] If $p^{3/4}\le |A| \le p$ and $p^{5/4}\le |B|\le p^{4/3}$, then 
\[\left\vert I(P, R)-\frac{|P||R|}{p^2}\right\vert\ll p^{1/16}|P|^{3/4}|R|^{1/2}|A|^{1/12}.\]

\item[\textup{(2)}] If $|A|\le p$ and $p\le |B|\le p^{5/4}$, then
\[\left\vert I(P, R)-\frac{|P||R|}{p^2}\right\vert\ll p^{3}|P|^{1/2}|R|^{1/2}\left(\frac{1}{p^{5}} + \frac{|P|^{1/3}|A|^{1/3}}{p^{17/3}} \right)^{1/2}.\]
\item[\textup{(3)}] If $|A|\le p$ and $|B|\le p$, then
\[\left\vert I(P, R)-\frac{|P||R|}{p^2}\right\vert\ll p^{3}|P|^{1/2}|R|^{1/2}\left(\frac{1}{p^{5}} + \frac{|P|^{2/3}}{p^{6}} \right)^{1/2}.\]
\end{enumerate}
\end{theorem}

\begin{theorem}[$|A|\ge p ~\mbox{\bf and}~|B|\le p^{4/3}$]\label{thm: incidences3}
Assume that $p \equiv 3 \mod{4}$ .
Let $P=A\times B$ for $A, B\subset \mathbb{F}_p^2$ with $|A|\le |B|$ and let $R$ be a set of rigid motions in $\mathbb{F}_p^2$. The following hold. 
\begin{enumerate}

\item[\textup{(1)}] If $p \le |A| \le p^{5/4}$ and $p \le |B| \le p^{5/4}$, then
\[\left\vert I(P, R)-\frac{|P||R|}{p^2}\right\vert\ll p^{1/3}|P|^{2/3}|R|^{1/2}.\]

\item[\textup{(2)}] If $p \le |A| \le p^{5/4}$ and $p^{5/4}\le |B|\le p^{4/3}$, then 
\[\left\vert I(P, R)-\frac{|P||R|}{p^2}\right\vert\ll p^{11/48}|P|^{2/3}|R|^{1/2}|B|^{1/12}.\]

\item[\textup{(3)}] If $p^{5/4}\le |A|\le p^{4/3}$ and $p^{5/4}\le |B|\le p^{4/3}$, then 
\[\left\vert I(P, R)-\frac{|P||R|}{p^2}\right\vert\ll p^{1/8}|P|^{3/4}|R|^{1/2}.\]
\end{enumerate}
\end{theorem}
\begin{theorem}[$|B|>p^{4/3}$]\label{thm: incidences32}
Assume that $p \equiv 3 \mod{4}$.
Let $P=A\times B$ for $A, B\subset \mathbb{F}_p^2$ with $|A|\le |B|$ and let $R$ be a set of rigid motions in $\mathbb{F}_p^2$. The following hold. 
\begin{enumerate}

\item[\textup{(1)}] If $p\le |A|\le p^{5/4}$ and $|B|>p^{4/3}$, then 
\[\left\vert I(P, R)-\frac{|P||R|}{p^2}\right\vert\ll p^{5/12}|P|^{5/8}|R|^{1/2}|A|^{1/24}.\]

\item[\textup{(2)}] If $p^{5/4}\le |A|\le p^{4/3}$ and $|B|>p^{4/3}$, then
\[\left\vert I(P, R)-\frac{|P||R|}{p^2}\right\vert\ll p^{5/16}|P|^{5/8}|R|^{1/2}|A|^{1/8}.\]
\end{enumerate}
\end{theorem}

We now include a comparison to the results in $\mathbb{F}_q^2$ offered by  Theorem \ref{thm: incidence1} and \cref{thm: incidences2}.

\vspace{5mm}
\begin{table}[!ht]
    \centering
   	\begin{tabular}{ |c|c|c|c|c|c|c|} 
 \hline
  & $|A|\le p^{\frac{1}{2}}$ & $p^{\frac{1}{2}}<|A|\le p^{\frac{3}{4}}$&$p^{\frac{3}{4}}< |A|\le p$& $p<|A|\le p^{\frac{5}{4}}$&$p^{\frac{5}{4}}< |A|\le p^{\frac{4}{3}}$ \\ 
 \hline
 $|B|\le p^{\frac{3}{4}}$ & $\varnothing$ & $\surd$ & $\surd$ &$\diagup $ &$\diagup $\\
 \hline 
 $p^{\frac{3}{4}}<|B|\le p$ & $\varnothing$ & $\surd$ & $\surd$& $\diagup $ &$\diagup $\\
 \hline
 $p<|B|\le p^{\frac{5}{4}}$ & $\varnothing$ & $\surd$& $\surd$& $\surd$&$\diagup $\\
 \hline
 $p^{\frac{5}{4}}<|B|\le p^{\frac{4}{3}}$ & $\varnothing$ & $\varnothing$& $\surd$&$\surd$&$\surd$\\
 \hline
 $p^{\frac{4}{3}}<|B|$ & $\varnothing$ & $\varnothing$ & $\varnothing$&$\surd$&$\surd$\\
 \hline 
\end{tabular}
\vspace{5mm}
    \caption{In this table, by $``\surd"$ we mean better result, by $``\varnothing"$ we mean weaker result, and by $``\diagup"$ we mean invalid range corresponding to $|B|\le |A|$. (Compared to \cref{thm: incidence1})}
    \label{tab:my_label2222}
\end{table}

\vspace{5mm}
\begin{table}[!ht]
    \centering
   	\begin{tabular}{ |c|c|c|c|c|c|c|} 
 \hline
  & $|A|\le p^{\frac{1}{2}}$ & $p^{\frac{1}{2}}<|A|\le p^{\frac{3}{4}}$&$p^{\frac{3}{4}}< |A|\le p$& $p<|A|\le p^{\frac{5}{4}}$&$p^{\frac{5}{4}}< |A|\le p^{\frac{4}{3}}$ \\ 
 \hline
 $|B|\le p^{\frac{3}{4}}$ & $\surd$ & $\surd$ & $\surd$ &$\diagup $ &$\diagup $\\
 \hline 
 $p^{\frac{3}{4}}<|B|\le p$ & $\surd$ & $\surd$ & $\surd$& $\diagup $ &$\diagup $\\
 \hline
 $p<|B|\le p^{\frac{5}{4}}$ & $\surd$ & $\surd$& $\surd$& $\surd$&$\diagup $\\
 \hline
 $p^{\frac{5}{4}}<|B|\le p^{\frac{4}{3}}$ & $\varnothing$ & $\varnothing$& $\surd$&$\surd$&$\surd$\\
 \hline
 $p^{\frac{4}{3}}<|B|$ & $\varnothing$ & $\varnothing$ & $\varnothing$&$|B|^3<p^2|A|^2$&$|B|<p^{\frac{3}{2}}$\\
 \hline 
\end{tabular}
\vspace{5mm}
    \caption{In this table, by $``\surd"$ we mean better result, by $``\varnothing"$ we mean weaker result, by $``f(|A|, |B|)"$ we mean better result under the condition $f(|A|, |B|)$, and by $``\diagup"$ we mean invalid range corresponding to $|B|\le |A|$. (Compared to \cref{thm: incidences2})}
\end{table}

 The incidence theorems in two dimensions provide both upper and lower bounds that simultaneously depend on the exponents of $p$, $|P|$, and $|R|$. This makes it challenging to formulate a conjecture that remains sharp across most parameter ranges.  

From the viewpoint of applications, for example Theorem~\ref{thm: growth d2p1} (3), Theorem~\ref{thm: incidences31} (3) is sharp in the sense that the term $p^{3}|P|^{1/2}|R|^{1/2}$ cannot be replaced by $p^{\,3-\epsilon_1}|P|^{\,1/2-\epsilon_2}|R|^{\,1/2-\epsilon_3}$ for any triple $(\epsilon_1,\epsilon_2,\epsilon_3)$ of non-negative numbers with $\epsilon_1+\epsilon_2+\epsilon_3>0$. Otherwise, by repeating the argument in the proof of Theorem~\ref{thm: growth d2p1} (3), one would deduce that $|A-gA|\gg |A|^{2+\delta}$ for some $\delta>0$ when $|A|$ is small, which is impossible.

\subsection{Proof of Theorems \ref{thm: incidence inequality} -- \ref{thm: incidences32}}
Let us present a framework that will work for most cases. 

We have 
\begin{align*}
 &I(P, R)=\sum_{\substack{(x, y)\in P, \\ (g, z)\in R}}1_{x=gy+z}=\frac{1}{q^d}\sum_{m\in \mathbb{F}_q^d}\sum_{\substack{(x, y)\in P,\\ (g, z)\in R}}\chi\left(m\cdot (x-gy-z)\right)\\
 &=\frac{|P||R|}{q^d}+\frac{1}{q^d}\sum_{m\in \mathbb{F}_q^d\setminus \{ 0 \}}\sum_{\substack{(x, y)\in P,\\ (g, z)\in R}}\chi(m\cdot(x-gy-z))\\
 &=\frac{|P||R|}{q^d}+q^d\sum_{m\ne 0}\sum_{(g, z)\in R}\widehat{P}(-m, gm)\chi(-mz)=:I+II,
    \end{align*}
where $\widehat{P}(u,v)=q^{-2d}\sum_{(x,y) \in P} \chi(-xu-yv)$. 
We next bound the second term. By the Cauchy-Schwarz inequality, we have 
\begin{align*}
  II&\le q^d|R|^{1/2}\left(\sum_{(g, z)\in O(d)\times \mathbb{F}_q^d}\sum_{m_1, m_2\ne 0} \widehat{P}(-m_1, gm_1)\overline{\widehat{P}(-m_2, gm_2)}\chi(z(-m_1+m_2))\right)^{1/2}\\
  &=q^d|R|^{1/2}\left(q^d\sum_{g\in O(d)}\sum_{m\ne 0}|\widehat{P}(m, -gm)|^2\right)^{1/2}\\
  &=q^{\frac{3d}{2}}|R|^{1/2}\left(\sum_{g\in O(d)}\sum_{m\ne 0}|\widehat{P}(m, -gm)|^2\right)^{1/2}.
\end{align*}
We now consider two cases. 

{\bf Case $1$:} If $P$ is a general set in $\mathbb{F}_q^{2d}$, then we have 
\[\sum_{g\in O(d)}\sum_{m\ne 0}|\widehat{P}(m, -gm)|^2\le |O(d-1)|\sum_{\substack{(m, m')\ne (0, 0), \\||m||=||m'||}}|\widehat{P}(m, m')|^2,\]
where we used the fact that the stabilizer of a non-zero element in $\mathbb{F}_q^d$ is at most $|O(d-1)|$. From here, we apply \cref{main-this-section} to obtain \cref{thm: incidence inequality}.

{\bf Case $2$:} If $P$ is of the structure $A\times B$, where $A, B\subset \mathbb{F}_q^d$, then 
\[\widehat{P}(m, -gm)=\widehat{A}(m)\widehat{B}(-gm).\]
Thus, 
\begin{align*}
    &\sum_{g\in O(d)}\sum_{m\ne 0}|\widehat{P}(m, -gm)|^2=\sum_{g\in O(d)}\sum_{m\ne 0}|\widehat{A}(m)|^2|\widehat{B}(-gm)|^2\\&\le |O(d-1)|\sum_{m\ne 0}|\widehat{A}(m)|^2\sum_{\substack{m'\ne 0,\\ ||m'||=||m||}}|\widehat{B}(m')|^2\\
    &\le |O(d-1)|\sum_{||m||=||m'||}|\widehat{A}(m)|^2|\widehat{B}(m')|^2,
\end{align*}
where we again used the fact that the stabilizer of a non-zero element in $\mathbb{F}_q^d$ is at most $|O(d-1)|$.

From here, we apply \cref{plan-ge-1}, \cref{key-2-d}, and \cref{th:novelty} to obtain \cref{thm: incidence1}, \cref{thm: incidences2}, \cref{thm: incidences31}, \cref{thm: incidences3}, and \cref{thm: incidences32}, except \cref{thm: incidences31} (2) and (3). 

To prove these two statements, we need to bound the sum $\sum_{\substack{g\in O(2), \\ m \ne 0}}|\widehat{P}(m, -gm)|^2$ in a different way. More precisely,
\begin{align*}
    &\sum_{\substack{g\in O(2), \\ m \ne 0}}|\widehat{P}(m, -gm)|^2=\frac{1}{p^{4}}\sum_{\substack{g\in O(2),\\ m \ne 0}}\sum_{x_1, y_1, x_2, y_2}P(x_1, y_1)P(x_2, y_2)\chi(m(x_1-gy_1-x_2+gy_2))\\
    &=\frac{1}{p^{4}}\left(p^2\sum_{g \in O(2)}\sum_{x_1, x_2, y_1, y_2}P(x_1, y_1)P(x_2, y_2)1_{x_1-x_2=g(y_1-y_2)}-|O(2)||P|^2\right).
\end{align*}
Moreover, 
\begin{align*}
    &\sum_{g \in O(2)}\sum_{x_1, x_2, y_1, y_2}P(x_1, y_1)P(x_2, y_2)1_{x_1-x_2=g(y_1-y_2)}\le |O(2)||P|+|N(P)|.\\
\end{align*}
The first approach is equivalent to bounding $N(P)$ by using \cref{framework} and \cref{th:novelty}.

If $|A|\le p$ and $p\le |B|\le p^{5/4}$, then \cref{th:novelty} tells us that 
$$I_{A, B}\ll (p|A|^2+|A|^{10/3})^{1/2}\cdot p^{1/3}|B|^{4/3}.$$
As a consequence, one has 
\[N(P)\le \frac{|A|^2|B|^2}{p}+(p|A|^2+|A|^{10/3})^{1/2}\cdot p^{1/3}|B|^{4/3}.\]
However, when $|A|\le p$ and $p\le |B|\le p^{5/4}$, by the Cauchy-Schwarz inequality, a better upper bound can be obtained. Indeed, using \cref{cor-smallsets}, we have
\[N(P)\le N(A\times A)^{1/2}N(B\times B)^{1/2}\ll p^{1/3}|A|^{5/3}|B|^{4/3}.\]
Together with the above estimates, we obtain
\[\sum_{\substack{g\in O(2), \\ m \ne 0}}|\widehat{P}(m, -gm)|^2\le \frac{|P|}{p^5}+\frac{|P|^{4/3}|A|^{1/3}}{p^{17/3}}.\]
This gives 
\[\left\vert I(P, R)-\frac{|P||R|}{p^2}\right\vert\ll p^{3}|P|^{1/2}|R|^{1/2}\left(\frac{1}{p^{5}} + \frac{|P|^{1/3}|A|^{1/3}}{p^{17/3}} \right)^{1/2}.\]
Similarly, if $|A|\le p$ and $|B|\le p$, we have 
\[N(P)\ll |A|^{5/3}|B|^{5/3},\]
and 
\begin{align*}
    \sum_{\substack{g\in O(2), \\ m \ne 0}}|\widehat{P}(m, -gm)|^2 & \ll  \frac{|P|}{p^{5}} + \frac{|P|^{5/3}}{p^{6}}.
\end{align*}
Hence,  
\[\left\vert I(P, R)-\frac{|P||R|}{p^2}\right\vert\ll p^{3}|P|^{1/2}|R|^{1/2}\left(\frac{1}{p^{5}} + \frac{|P|^{2/3}}{p^{6}} \right)^{1/2}.\]

\subsection{Sharpness of \cref{thm: incidence inequality} and \cref{thm: incidence1}}\label{subsection-inci-shapr}

We first show that \cref{thm: incidence inequality} is sharp up to a constant factor.

Let $X$ be an arithmetic progression in $\mathbb{F}_q$, and let $v_1, \ldots, v_{\frac{d-1}{2}}$ be $(d-1)/2$ vectors in $\mathbb{F}_q^{d-1}\times \{0\}$ such that $v_i\cdot v_j=0$ for all $1\le i\le j\le (d-1)/2$. The existence of such vectors can be found in Lemma 5.1 in \cite{TAMS} when ($d=4k+1$) or ($d=4k+3$ with $q\equiv 3\mod 4$). Define 
\[A=B=\mathbb{F}_q\cdot v_1+\cdots+\mathbb{F}_q\cdot v_{\frac{d-1}{2}}+X\cdot e_d,\]
here $e_d=(0, \ldots, 0, 1)$. Set $P=A\times B$. The number of quadruples $(x, y, u, v)\in A\times A\times B\times B$ such that $||x-y||=||u-v||$, is at least a constant times $|X|^2q^{2d-1}$, say, $|X|^2q^{2d-1}/2$. For each $(g, z)\in O(d)\times \mathbb{F}_q^d$, let $i(g, z)=\#\{(u, v)\in A\times B\colon gu+z=v\}$. Define $\mathcal{Q}=\sum_{(g, z)}i(g, z)^2$. So, $\mathcal{Q}\ge |X|^2q^{2d-1}|O(d-1)|/2$.

We call $g$ $\textbf{type-k}$, $0\le k\le (d-1)/2$, if the rank of the system $$\left\lbrace v_1, \ldots, v_{(d-1)/2}, e_d, gv_1, \ldots, gv_{(d-1)/2}\right\rbrace$$ is $d-k$. 

For any pair $(g, z)$, where $g$ is $\textbf{type-0}$, the number of $(u, v)\in A\times B$ such that $gu+z=v$ is at most $|X|$.

For $0< k\le (d-1)/2$, if $g$ is $\textbf{type-k}$, then, assume, 
\[gv_1, \ldots, gv_k\in \mathtt{Span}(gv_{k+1}, \ldots, gv_{\frac{d-1}{2}}, v_1, \ldots, v_{\frac{d-1}{2}}, e_d).\]
Let $N(k)$ be the contribution to $\mathcal{Q}$ of pairs $(g, z)$ such that $g$ is $\textbf{type-k}$. Then, $N(k)$ is at most the number of $\textbf{type-k}$ $g$s times $|A|^2$. For each $k$, to count the number of $\textbf{type-k}$ $g$s, we observe that $||v_i||=0$, so $||gv_i||=0$. The number of elements of norm zero in $\mathtt{Span}(gv_{k+1}, \ldots, gv_{\frac{d-1}{2}}, v_1, \ldots, v_{\frac{d-1}{2}}, e_d)$ is at most $q^{d-k}$. So, the total number of $\textbf{type-k}$ $g$s such that 
\[gv_1\in \mathtt{Span}(gv_{k+1}, \ldots, gv_{\frac{d-1}{2}}, v_1, \ldots, v_{\frac{d-1}{2}}, e_d)\]
is at most $q^{d-k}|O(d-1)|$, which is, of course, larger than the number of $g$ satisfying  
\[gv_1, \ldots, gv_k\in \mathtt{Span}(gv_{k+1}, \ldots, gv_{\frac{d-1}{2}}, v_1, \ldots, v_{\frac{d-1}{2}}, e_d).\]
Summing over all $k\ge 1$ and the corresponding $\textbf{type-k}$ $g$s, the contribution to $\mathcal{Q}$ is at most $|X|^2q^{d-1}q^{d-k}|O(d-1)|\le |X|^2q^d|O(d)|q^{-k}$. So, the pairs $(g, z)$, where $g$ is $\textbf{type-k}$ and $k\ge 1$, contribute at most $\ll |X|^2q^{d-1}|O(d-1)|$ which is much smaller than $\mathcal{Q}/2$. Thus, we can say that the contribution of $\mathcal{Q}$ mainly comes from $\textbf{type-0}$ $g$s. 

Let $R$ be the set of pairs $(g, z)\in O(d)\times \mathbb{F}_q^d$ such that $i(g, z)\ge 2$ and $g$ is $\textbf{type-0}$. 

Whenever $|X|=cq$, $0<c<1$, by a direct computation, \cref{thm: incidence inequality} shows that 
\[I(P, R)\le Cq^{\frac{d^2-d+2}{4}}\sqrt{|P||R|}=Cq^{\frac{d^2+d}{4}}|X|\sqrt{|R|}\le C|X||O(d)|q^d,\]
for some positive constant $C$. This gives $\mathcal{Q}\le C|X|^2|O(d)|q^d$. This matches the lower bound of $|X|^2q^{d}|O(d)|/2$ up to a constant factor.

We note that this example can also be used to show the sharpness of \cref{thm: incidence1}(2) in the same way. 

For the sharpness of \cref{thm: incidence1}(1), let $X\subset \mathbb{F}_q$ with $|X|=cq$, $0<c<1$.
Set 
\[A=X\cdot v_1+\cdots+X\cdot v_{\frac{d-1}{2}},~B=X\cdot v_1+\cdots+X\cdot v_{\frac{d-1}{2}}+X\cdot e_d,\]
where $e_d=(0, \ldots, 0, 1)$.
Since any vector in $A-gB$ is of the form 
\[-g(x_1v_1+\cdots+x_{\frac{d-1}{2}}v_{\frac{d-1}{2}})+y_1v_1+\cdots+y_{\frac{d-1}{2}}v_{\frac{d-1}{2}}-x_{\frac{d+1}{2}}ge_d,\]
where $x_i, y_i\in X$, we have $|A-gB|\le |X|^{d}\le c^{d}q^d$ for all $g\in O(d)$. Let $R$ be the set of $(g, z)$ such that $z\in A-gB$. Then, we have $|R|\le c^dq^d|O(d)|$. \cref{thm: incidence1}(1) gives 
\[I(P, R)\le Cq^{\frac{d^2-d}{4}}\sqrt{|P||R|}\le Cc^dq^{\frac{d^2+d}{2}},\]
for some positive constant $C$. 

On the other hand, by the definitions of $A$, $B$, and $R$, we have 
\[I(P, R)=|O(d)||P|=c^dq^d|O(d)|=c^dq^{\frac{d^2+d}{2}}.\]
This matches the incidence bound up to a constant factor.
\subsection{Some discussions}
In this subsection, we present direct incidence bounds which can be proved by using the Cauchy-Schwarz inequality and the results on $N(P)$ from the previous section.
\begin{theorem}
Let $P$ be a set of points in $\mathbb{F}_q^d\times \mathbb{F}_q^d$ and $R$ be a set of rigid motions in $\mathbb{F}_q^d$. Then we have 
\[I(P, R)\ll |R|^{1/2}|O(d-1)|^{1/2}\left(\frac{|P|^2}{q}+Cq^d|P|\right)^{1/2}+|R|,\]
for some absolute constant $C>0$.
\end{theorem}
\begin{proof}
For each $r\in R$, denote $I(P, r)$ by $i(r)$. Then, it is clear that 
\begin{equation}\label{eq:C-Spr}I(P, R)=\sum_{r\in R}i(r)\le |R|^{1/2}\left(\sum_{r\in R}i(r)^2\right)^{1/2}.\end{equation}
We observe that, for each $r \in R$, $i(r)^2$ counts the number of pairs $(a_1, b_1), (a_2, b_2)\in P$ on $r$. This implies $||a_1-a_2||=||b_1-b_2||$. Thus, the sum $\sum_{r\in R}i(r)^2$ can be bounded by 
\[|O(d-1)|N(P) + I(P, R),\]
where we used the fact that the stabilizer of a non-zero element is at most $|O(d-1)|$, and the term $I(P, R)$ comes from pairs $(a_1, b_1), (a_2, b_2)\in P$ with $a_1=a_2$ and $b_1=b_2$. 
Therefore, 
\[\sum_{r\in R}i(r)^2\ll |O(d-1)|N(P)+I(P, R).\]
Using \cref{quadruple}, the theorem follows. 
\end{proof}
If we use the trivial bound $N(P)\le |P|^2$, then the next theorem is obtained. 
\begin{theorem}
Let $P$ be a set of points in $\mathbb{F}_q^d\times \mathbb{F}_q^d$ and $R$ be a set of rigid motions in $\mathbb{F}_q^d$. Then we have 
\[I(P, R)\ll |P||R|^{1/2}|O(d-1)|^{1/2}+|R|.\]
\end{theorem}
Compared to \cref{thm: incidence inequality} and \cref{thm: incidence1}, these two incidence theorems only give weaker upper bounds and tell us nothing about the lower bounds. 

In two dimensions over prime fields, if $P=A\times B$ with $|A|, |B|\le p$, then \cref{cor-smallsets} says that $N(P)\ll |A|^{5/3}|B|^{5/3}$. As above, the next theorem is a direct consequence. 
\begin{theorem}\label{thm-small-set-p}
    Let $P=A\times B$ with $A, B\subset \mathbb{F}_p^2$ and $p\equiv 3\mod 4$. Assume that $|A|, |B|\le p$, then we have 
\[I(P, R)\ll |P|^{5/6}|R|^{1/2}+|R|.\]
In particular, if $|P|=|R|=N$ then 
\[I(P, R)\ll N^{4/3}.\]
\end{theorem}

\section{Intersection pattern I}\label{subsec: 2}
We recall Question 1.1.
\ques*

We note that for any $g\in O(d)$, we have 
\[\sum_{z\in \mathbb{F}_q^d}|A\cap (g(B)+z)|=|A||B|.\]
Thus, there always exists $z\in \mathbb{F}_q^d$ such that $|A\cap (g(B)+z)|\ge \frac{|A||B|}{q^d}$.

We begin with some examples.

{\bf Example 1:} Let $A$ be a subspace of dimension $k$ with $1\le k< d$, $A=B$, and $\mathtt{Stab}(A)$ be the set of matrices $g\in O(d)$ such that $gA=A$. It is well-known that $|\mathtt{Stab}(A)|\sim |O(d-k)|\sim q^{\binom{d-k}{2}}$. For any $g\in \mathtt{Stab}(A)$, we have
\[|A\cap (gB+z)|=\begin{cases}0~&\mbox{if}~z\not\in A\\
|A|~&\mbox{if}~z\in A\end{cases}.\]

{\bf Example 2:} In $\mathbb{F}_q^d$ with $d$ odd, given $0<c<1$, for each $g\in O(d)$, there exist $A, B\subset \mathbb{F}_q^d$ with $|A|=|B|=cq^{\frac{d+1}{2}}$ and $|A-gB|\le 2cq^d$. To see this, let $A= \mathbb{F}_q^{\frac{d-1}{2}}\times \{0\}^{\frac{d-1}{2}}\times X$ where $X$ is an arithmetic progression of size $cq$ and $B=g^{-1}\left( \mathbb{F}_q^{\frac{d-1}{2}} \times \{0\}^{\frac{d-1}{2}}\times X\right)$. It is clear that the number of $z$ such that $A\cap (gB+z)\ne\emptyset$ is at most $|A-gB|\le q^{d-1}|X-X|\le 2cq^d$.

These two examples suggest that if we want 
\[|A\cap (g(B)+z)| \sim \frac{|A||B|}{q^d},\]
for almost every $z\in \mathbb{F}_q^d$, then $A$ and $B$ cannot be small and $g$ cannot be chosen arbitrarily in $O(d)$. The first theorem describes this phenomenon in detail.
\begin{theorem}\label{og}
 Let $A$ and $B$ be sets in $\mathbb{F}_q^d$. Then there exists $E\subset O(d)$ with $$|E|\ll \frac{|O(d-1)|q^{2d}}{|A||B|},$$ such that for any $g\in O(d)\setminus E$, there are at least $\gg q^d$ elements $z$ satisfying $$|A\cap (g(B)+z)| \sim \frac{|A||B|}{q^d}.$$ 
\end{theorem}
This theorem is valid in the range $|A||B|\gg q^{d+1}$, since
\[\frac{|O(d-1)|q^{2d}}{|A||B|}\ll |O(d)| ~~\mbox{when}~~|A||B|\gg q^{d+1}.\]
The condition $|A||B|\gg q^{d+1}$ is sharp in odd dimensions for comparable sets $A$ and $B$. A construction will be provided in \cref{sub-4142}. When one set is of small size, we can hope for a better estimate, and the next theorem presents such a result.
\begin{theorem}\label{thm: int2}
 Let $A$ and $B$ be sets in $\mathbb{F}_q^d$. Assume in addition either ($d\ge 3$ odd) or ($d\equiv 2\mod 4 ~\text{ and }~ q\equiv 3\mod 4$). Then there exists $E\subset O(d)$ such that for any $g\in O(d)\setminus E$, there are at least $\gg q^d$ elements $z$ satisfying $$|A\cap (g(B)+z)| \sim \frac{|A||B|}{q^d}.$$ In particular, 
\begin{enumerate}
    \item[\textup{(1)}] If $|A|<q^{\frac{d-1}{2}}$, then one has $|E|\ll \frac{q^{(d^{2}+d)/2}}{|A||B|}$.
    \item[\textup{(2)}] If $q^{\frac{d-1}{2}}\le |A|\le q^{\frac{d+1}{2}}$, then one has
$|E|\ll \frac{q^{(d^{2}+1)/2}}{|B|}$.
\end{enumerate}
\end{theorem}

Theorem \ref{thm: int2} is better than Theorem \ref{og} under one of the following conditions:
\begin{enumerate}
    \item[\textup{(1)}] $|A|\le q^{\frac{d-1}{2}}$ and $|A||B|\gg q^d$;
    \item[\textup{(2)}] $q^{\frac{d-1}{2}}\le |A|\le q^{\frac{d+1}{2}}$ and $|B|\gg q^{\frac{d+1}{2}}$.
\end{enumerate}

The proof involves a number of results from Restriction theory in which the conditions on $d$ and $q$ are necessary. We do not know if it is still true for the case $d\equiv 2\mod 4$ and $q\equiv 1\mod 4$, so it is left as an open question. 

The sharpness construction of \cref{og} can be modified to show that these two statements are also optimal in odd dimensions. In even dimensions, there is no evidence to believe that the above theorems are sharp. In the next theorem, we present an improvement in two dimensions when $|A|\ge q$.
\begin{theorem}\label{thm: int for d=2}
Assume that $q \equiv 3 \mod{4}$. Let $A$ and $B$ be sets in $\mathbb{F}_q^2$ with $|A| \le |B|$. Then there exists $E\subset O(2)$ with $$|E|\ll \frac{q^{3}}{|A|^{1/2}|B|}$$ such that for any $g\in O(2)\setminus E$, there are at least $\gg q^2$ elements $z$ satisfying $$|A\cap (g(B)+z)| \sim \frac{|A||B|}{q^2}.$$ 
\end{theorem}

Note that, in both Theorems \ref{thm: int2} and \ref{thm: int for d=2}, the roles of $A$ and $B$ are symmetric, so similar results hold with the roles of $A$ and $B$ interchanged when $|B|\le |A|$.

If the sets lie in the plane over a prime field $\mathbb{F}_p$, then by employing recent \(L^2\) distance bounds, further improvements can be made when either $A$ is of medium size $(p\le |A|\le p^{4/3})$ or $B$ is of large size $(|B|\ge p^{4/3})$.

\begin{theorem}[$\mbox{\bf Medium} ~A$]\label{thm: int for d=2 p}
Assume that $p\equiv 3\mod{4}$. Let $A$ and $B$ be sets in $\mathbb{F}_p^2$ with $|A| \le |B|$. 
Then there exists $E\subset O(2)$ such that for any $g\in O(2)\setminus E$, there are at least $\gg p^2$ elements $z$ satisfying $$|A\cap (g(B)+z)| \sim \frac{|A||B|}{p^2}.$$
In particular,
\begin{enumerate}
\item[\textup{(1)}] If $p \le |A| \le p^{5/4}$ and $p^{5/4}\le |B|\le p^{4/3}$, then 
$|E|\ll \frac{p^{59/24}}{|A|^{2/3}|B|^{1/2}}.$
\item[\textup{(2)}] If $p^{5/4}\le |A|\le p^{4/3}$ and $p^{5/4}\le |B|\le p^{4/3}$, then $|E|\ll \frac{p^{9/4}}{|A|^{1/2}|B|^{1/2}}.$
\end{enumerate}
\end{theorem}

\begin{theorem}[$\mbox{\bf Large} ~B$]\label{thm: int for d=2 p11} 
Assume that $p\equiv 3\mod{4}$.
Let $A$ and $B$ be sets in $\mathbb{F}_p^2$ with $|A| \le |B|$. 
Then there exists $E\subset O(2)$ such that for any $g\in O(2)\setminus E$, there are at least $\gg p^2$ elements $z$ satisfying $$|A\cap (g(B)+z)| \sim \frac{|A||B|}{p^2}.$$
In particular,
\begin{enumerate}
\item[\textup{(1)}] If $p\le |A|\le p^{5/4}$ and $|B|>p^{4/3}$, then 
$|E|\ll \frac{p^{17/6}}{|A|^{2/3}|B|^{3/4}}.$

\item[\textup{(2)}] If $p^{5/4}\le |A|\le p^{4/3}$ and $|B|>p^{4/3}$, then
$|E|\ll \frac{p^{21/8}}{|A|^{1/2}|B|^{3/4}}.$
\end{enumerate}
\end{theorem}

In the above two theorems, the sets $A$ and $B$ cannot be both small since, otherwise, it might lead to a contradiction from the inequalities that $p^2\ll |A-gB|\ll |B|^2$. To compare to the Theorem \ref{thm: int for d=2} over arbitrary finite fields, we include the following table.

\vspace{2mm}
\begin{table}[!ht]
    \centering
   	\begin{tabular}{ |c|c|c|c|c|c|c|} 
 \hline
  & $|A|\le p^{\frac{1}{2}}$ & $p^{\frac{1}{2}}<|A|\le p^{\frac{3}{4}}$&$p^{\frac{3}{4}}< |A|\le p$& $p<|A|\le p^{\frac{5}{4}}$&$p^{\frac{5}{4}}< |A|\le p^{\frac{4}{3}}$ \\ 
 \hline
 $|B|\le p^{\frac{3}{4}}$ & $\diagup $ & $\diagup $ & $\diagup $ &$\diagup $ &$\diagup $\\
 \hline 
 $p^{\frac{3}{4}}<|B|\le p$ & $\diagup $ & $\diagup $ & $\diagup $& $\diagup $ &$\diagup $\\
 \hline
 $p<|B|\le p^{\frac{5}{4}}$ & $\diagup $ & $\diagup $ & $\mathtt{unknown}$& $\varnothing$&$\diagup $\\
 \hline
 $p^{\frac{5}{4}}<|B|\le p^{\frac{4}{3}}$ & $\diagup $ & $\mathtt{unknown}$& $\mathtt{unknown}$&$\surd$&$\surd$\\
 \hline
 $p^{\frac{4}{3}}<|B|$ & $\varnothing$ & $\varnothing$ & $\varnothing$&$|B|^3<p^2|A|^2$&$|B|<p^{\frac{3}{2}}$\\
 \hline 
\end{tabular}
\vspace{5mm}
    \caption{In this table, by $``\surd"$ we mean better result, by $``\varnothing"$ we mean weaker result, by $``f(|A|, |B|)"$ we mean better result under the condition $f(|A|, |B|)$, by ``$\mathtt{unknown}$" we mean there is no nontrivial result in this range yet, and by $``\diagup"$ we mean invalid range corresponding to $|B|\le |A|$ or $|A||B|\ll p^2$.}
    \label{tab:my_label}
\end{table}
We propose a conjecture in even dimensions, which will be stated explicitly in the next section once the connections with distance problems have been made precise.

\subsection{Proof of Theorems \ref{og} -- \ref{thm: int for d=2 p11}}\label{sec6}
Theorem \ref{mot} follows immediately from Theorems \ref{og}, \ref{thm: int for d=2}, and \ref{thm: int for d=2 p}. 

\begin{proof}[Proof of \textup{\cref{og}}]
Set 
\begin{align*}
E_{1}&=\left\lbrace g \in O(d)\colon \#\{z \in \mathbb{F}_{q}^{d}\colon |A\cap (g(B)+z)|\le \frac{|A||B|}{2q^d}\}\ge cq^d\right\rbrace \\
E_{2}&=\left\lbrace g \in O(d)\colon \#\{z \in \mathbb{F}_{q}^{d}\colon |A\cap (g(B)+z)|\ge \frac{3|A||B|}{2q^d}\}\ge cq^d\right\rbrace
\end{align*}
for some $c \in (0,1)$.

We first show that $|E_{1}|\ll \frac{|O(d-1)|q^{2d}}{c|A||B|}$. Indeed, let $R_{1}$ be the set of pairs $(g, z)$ with $g\in E_{1}$ and $|A\cap (g(B)+z)|\le |A||B|/2q^d$. It is clear that $I(A\times B, R_{1})\le \frac{|A||B||R_{1}|}{2q^d}$. 
On the other hand, \cref{thm: incidence inequality} also tells us that 
\[I(A\times B, R_{1})\ge \frac{|A||B||R_{1}|}{q^d}-C|O(d-1)|^{1/2}q^{d/2}\sqrt{|A||B||R_{1}|},\]
for some positive constant $C$.
Thus, we have 
\[|R_{1}|\le \frac{4C^2|O(d-1)|q^{3d}}{|A||B|}.\]
Together this with $|R_{1}|\ge c|E_{1}|q^d$ implies the desired conclusion.

Similarly, for $E_{2}$, let $R_{2}$ be the set of pairs $(g,z)$ with $g \in E_{2}$ and $|A\cap (g(B)+z)|\ge 3|A||B|/2q^d$.
Then, $I(A\times B, R_{2})\ge \frac{3|A||B||R_{2}|}{2q^d}$.
By \cref{thm: incidence inequality} again, we obtain
\[I(A\times B, R_{2})\le \frac{|A||B||R_{2}|}{q^d}+C|O(d-1)|^{1/2}q^{d/2}\sqrt{|A||B||R_{2}|},\]
and thus
\[|R_{2}|\le \frac{4C^2|O(d-1)|q^{3d}}{|A||B|}.\]
The fact $|R_{2}| \ge c |E_{2}|q^{d}$ implies the desired result $|E_{2}|\ll \frac{|O(d-1)|q^{2d}}{c|A||B|}$.

Next, note that for any $g \in O(d) \setminus ( E_{1} \cup E_{2} )$, by setting $c=1/4$, we have 
\begin{align*}
&\#\{z \in \mathbb{F}_{q}^{d}\colon |A\cap (g(B)+z)|\ge \frac{|A||B|}{2q^d}\}\ge \frac{3}{4}q^d \\
&\#\{z \in \mathbb{F}_{q}^{d}\colon |A\cap (g(B)+z)|\le \frac{3|A||B|}{2q^d}\}\ge \frac{3}{4}q^d,
\end{align*}
implying that there are at least $q^{d}/2$ elements $z$ satisfying
\[ \frac{|A||B|}{2q^d} \le \#\{z \in \mathbb{F}_{q}^{d}\colon |A\cap (g(B)+z)|\}\le \frac{3|A||B|}{2q^d}.\]
\end{proof}

\begin{proof}[Proof of \textup{\cref{thm: int2}}]
We consider the case when $|A|<q^{\frac{d-1}{2}}$, since other cases can be treated in the same way.
We use the same notations as in \cref{og}. 
By \cref{thm: incidence1}, there exists $C>0$ such that
\[I(A\times B, R_{1}) \ge \frac{|A||B||R_{1}|}{q^d}-Cq^{(d^{2}-d)/4}\sqrt{|A||B||R_{1}|},\]
and 
\[I(A\times B, R_{2})\le \frac{|A||B||R_{2}|}{q^d}+Cq^{(d^{2}-d)/4}\sqrt{|A||B||R_{2}|}.\]
This, together with $|R_{1}| \ge c|E_{1}|q^{d}$ and $|R_{2}| \ge c|E_{2}|q^{d}$, implies that 
\[|E_{1}|,|E_{2}| \ll \frac{q^{(d^{2}+d)/2}}{|A||B|}.\]
Similarly to the proof of \cref{og}, the theorem follows.
\end{proof}
\begin{proof}[Proof of \textup{\cref{thm: int for d=2}}]
We use the same notations as in \cref{og} for $d=2$. 
By \cref{thm: incidences2}, there exists $C>0$ such that
\[I(A\times B,R_{1}) \ge \frac{|A||B||R_{1}|}{q^{2}} - Cq^{1/2}\sqrt{|A||B||R_{1}|} \,|A|^{1/4},\]
and 
\[I(A\times B,R_{2})\le \frac{|A||B||R_{2}|}{q^{2}} + Cq^{1/2}\sqrt{|A||B||R_{2}}| \,|A|^{1/4}.\]
This, together with $|R_{1}| \ge c|E_{1}|q^{2}$ and $|R_{2}| \ge c|E_{2}|q^{2}$, gives that
\[|E_{1}|,|E_{2}| \ll \frac{q^{3}}{|A|^{1/2}|B|}.\]
As in the proof of \cref{og}, the theorem follows.
\end{proof}
Using \cref{thm: incidences3} and \cref{thm: incidences32} respectively, the proofs of  \cref{thm: int for d=2 p} and \cref{thm: int for d=2 p11} can be obtained in the same way.

\subsection{Sharpness of \cref{og} and \cref{thm: int2}}\label{sub-4142}

The constructions we present here are similar to those in the previous section. For completeness, we provide the details below.

It follows from the proof of \cref{og} that for any $0<c<1$, there exists $C=C(c)$ such that if $|A||B|\ge Cq^{d+1}$, then there are at least $(1-c)|O(d)|q^d$ pairs $(g, z)\in O(d)\times \mathbb{F}_q^d$ such that 
\begin{equation}\label{appro}\frac{|A||B|}{2q^d}\le |A\cap (g(B)+z)|\le \frac{3|A||B|}{2q^d}.\end{equation}
We now show that this result is sharp in the sense that for $0<c<1$ small enough, say, $8c<1/9(1-c)$, there exist $A, B\subset \mathbb{F}_q^d$ with $|A|=|B|=|X|q^{\frac{d-1}{2}}$, $cq<|X|<1/9(1-c)q$, such that the number of 
pairs $(g, z)\in O(d)\times \mathbb{F}_q^d$ satisfying \cref{appro} is at most $(1-c)|O(d)|q^d$. 

To be precise, let $X$ be an arithmetic progression in $\mathbb{F}_q$, and let $v_1, \ldots, v_{\frac{d-1}{2}}$ be $(d-1)/2$ vectors in $\mathbb{F}_q^{d-1}\times \{0\}$ such that $v_i\cdot v_j=0$ for all $1\le i\le j\le (d-1)/2$. The existence of such vectors can be found in Lemma 5.1 in \cite{TAMS} when ($d=4k+1$) or ($d=4k+3$ with $q\equiv 3\mod 4$). Define 
\[A=B=\mathbb{F}_q\cdot v_1+\cdots+\mathbb{F}_q\cdot v_{\frac{d-1}{2}}+X\cdot e_d,\]
here $e_d=(0, \ldots, 0, 1)$. 

We first note that the distance between two points in $A$ or $B$ is of the form $(x-x')^2$. By a direct computation and the fact that $X$ is an arithmetic progression, the number of quadruples $(x, y, u, v)\in A\times A\times A\times A$ such that $||x-y||=||u-v||$, is at least a constant times $|X|^3q^{2d-2}$, say, $|X|^3q^{2d-2}/2$. For each $(g, z)\in O(d)\times \mathbb{F}_q^d$, let $i(g, z)=\#\{(u, v)\in A\times B\colon gu+z=v\}$. Define $\mathcal{Q}=\sum_{(g, z)}i(g, z)^2$. So, $\mathcal{Q}\ge |X|^3q^{2d-2}|O(d-1)|/2$.

We note that $|A|=|B|=|X|q^{\frac{d-1}{2}}$. If there were at least $(1-c)|O(d)|q^d$ pairs $(g, z)\in O(d)\times \mathbb{F}_q^d$ satisfying \cref{appro},
then we would bound $\mathcal{Q}$ in a different way, which leads to a bound that much smaller than $|X|^3|O(d-1)|q^{2d-1}$, so we have a contradiction. 

Choose $(1-c)|O(d)|q^d$ pairs $(g, z)$ satisfying \cref{appro}. The contribution of these pairs is at most $\frac{9(1-c)}{4}|O(d)|q^d\frac{|X|^4}{q^2}$ to $\mathcal{Q}$. 

The number of remaining pairs $(g, z)$ is at most $c|O(d)|q^d$. We now compute the contribution of these pairs to $\mathcal{Q}$. As before, we call $g$ $\textbf{type-k}$, $0\le k\le (d-1)/2$, if the rank of the system $\{v_1, \ldots, v_{(d-1)/2}, e_d, gv_1, \ldots, gv_{(d-1)/2}\}$ is $d-k$. 

For any pair $(g, z)$, where $g$ is $\textbf{type-0}$, the number of $(u, v)\in A\times B$ such that $gu+z=v$ is at most $|X|$. Thus, the contribution of pairs with $\textbf{type-0}$ $g$s is at most $c|O(d)|q^d|X|^2$.  

As before, the contribution to $\mathcal{Q}$ of $\textbf{type-k}$ $g$s, with $k\ge 1$, is at most $|X|^2q^{d-1}q^{d-k}|O(d-1)|\le |X|^2q^d|O(d)|q^{-k}$.  
In other words, we have 
\begin{align*}
&\mathcal{Q}\le \frac{9(1-c)}{4}|O(d)|q^d\frac{|X|^4}{q^2}+c|O(d)|q^d|X|^2+\sum_{k=1}^{\frac{d-1}{2}}|X|^2|O(d)|q^d\frac{1}{q^k}\\
&=\frac{9(1-c)}{4}|O(d)|q^d\frac{|X|^4}{q^2}+2c|O(d)|q^d|X|^2,
\end{align*}
when $q$ is large enough. 
By choosing $c$ small enough, we see that $\mathcal{Q}<|X|^3q^{2d-2}|O(d-1)|/2$, a contradiction. 

The second statement of \cref{thm: int2} is valid in the range $|B|\gg q^{\frac{d+1}{2}}$. This is also sharp in odd dimensions by the above construction, since we can choose $|A|=|B|<q^{\frac{d+1}{2}}$ and the conclusion fails. 

The first statement of \cref{thm: int2} is valid in the range $|A||B|\gg q^{d}$. This is clear, since if $|A||B|\sim q^{d-\epsilon}$ for some positive $\epsilon$, then $|A-gB|\le q^{d-\epsilon}$ for all $g$.

\section{On a question of Mattila and the Erd\H{o}s-Falconer distance problem}
As mentioned earlier, in two dimensions Mattila asked in his survey paper \cite{Mat23} whether the condition  $\dim_H(A)=\dim_H(B)>\frac{5}{4}$ would be enough to guarantee that $\mathcal{L}^2(gA-B)>0$ by using the techniques developed by Guth, Iosevich, Ou, and Wang in \cite{alex-fal} on the Falconer distance problem. 

Theorem \ref{thm: int for d=2 p} (2) resolves the prime field version of this question in a stronger form as follows.
\begin{theorem}\label{hai}
Assume that $p\equiv 3\mod{4}$. Let $A$ and $B$ be sets in $\mathbb{F}_p^2$ with $|A|, |B|\ge p^{5/4}$. Then for almost every $g\in O(2)$, there exist at least $\gg p^2$ elements $z$ such that $$|A\cap (g(B)+z)|\sim \frac{|A||B|}{p^2}.$$ 
\end{theorem}
In light of this result, Mattila's question can be restated as: for Borel sets $A$ and $B$ in $\mathbb{R}^2$ of Hausdorff dimension $s$ with $s>\frac{5}{4}$, and assume in addition that the Hausdorff measures satisfy $\mathcal{H}^{\frac{5}{4}}(A)>0$ and $\mathcal{H}^{\frac{5}{4}}(B)>0$, then, for almost every $g\in O(2)$, one has 
    \[\mathcal{L}^2\left(\left\lbrace z\in \mathbb{R}^2\colon \dim_H(A\cap (z-gB))\ge 2s-2 \right\rbrace\right)>0.\]
In the prime field setting, the exponent $\textbf{5/4}$ can be improved to $\textbf{1}$ with a weaker conclusion, namely, $|A\cap (g(B)+z)|\ge 1$ instead of $|A\cap (g(B)+z)|\sim \frac{|A||B|}{p^2}.$

\begin{theorem}\label{thm0.3}
Assume that $p\equiv 3\mod{4}$. Let $A$ and $B$ be sets in $\mathbb{F}_p^2$ with $|A|, |B|\ge p$. 
Then for almost every $g\in SO(2)$, there are at least $\gg p^2$ elements $z$ satisfying $$|A\cap (g(B)+z)| \ge 1.$$
\end{theorem}

The conditions $|A|, |B|\ge p$ are optimal and cannot be improved to $p^{1-\epsilon}$ for any positive $\epsilon$, as the number of such $z$ is at most $|A||B|$.

As a consequence, we prove the Rotational Erd\H{o}s–Falconer problem in the plane. 

For the reader's convenience, we recall the Erd\H{o}s–Falconer distance conjecture, which states that for any set $A$ in $\mathbb{F}_q^d$, if $d\ge 2 $ is even and $|A|\gg q^{\frac{d}{2}}$, then the distance set $\Delta(A)$ covers a positive proportion of all distances. In two dimensions, the best current exponents are $\frac{5}{4}$ and $\frac{4}{3}$ in prime fields \cite{MPPRS} and arbitrary finite fields \cite{chap}, respectively. In odd dimensions, it has been proved that in \cite{TAMS} that the exponent $\frac{d+1}{2}$ is optimal.

\begin{theorem}\label{rational}
Assume that $p\equiv 3\mod{4}$.
    Let $A, B\subset \mathbb{F}_p^2$ with $|A|, |B|\ge p$. Then for almost every $g\in SO(2)$, one has $|\Delta(A, gB)|\gg p$.
\end{theorem}
Here $SO(2)$ is the set of rotations in $O(2)$, i.e. the set of orthogonal matrices with determinant $1$. It is well-known that $|SO(2)|=p+1$.

This theorem is sharp. 
For any $c\in (0, 1)$, in Proposition \ref{proposition7.2}, we show that there exist sets $A, B\subset \mathbb{F}_p^2$ with $|A|=|B|=M\sim p^{1-c}$ such that for all $g\in SO(2)$, one has $|\Delta(A, gB)|\le M$. To construct this example, we have used the properties that
\begin{itemize}
\item $\mathbb{F}_p^2 \cong \mathbb{F}_{p^2}$, under $p\equiv 3\mod 4$;
\item Rotations correspond to multiplication by unit norm elements.
\end{itemize}
These two properties give us algebraic control over all rotations simultaneously. However, in higher dimensions, it is hard to construct such an example. In fact, it is not hard to have examples with $|A|=|B|\sim p^{\frac{d}{2}}$ in even dimensions, $|A|=|B|\sim p^{\frac{d+1}{2}}$ in odd dimensions, and $o(|SO(d)|)$ of matrices $g\in SO(d)$ such that $|A-gB|=o(p^d)$. 

The condition $p \equiv 3 \mod{4}$ is necessary in the statement of Theorem \ref{rational}. Indeed, if $p \equiv 1 \mod{4}$, then there exist two isotropic lines through the origin in $\mathbb{F}_p^2$, which we denote by $\ell_1$ and $\ell_2$. Let $A$ be the set of $p$ points on $\ell_1$. Then, $\Delta(A, gA) = \{0\},$ since $gA = A$ for all $g \in SO(2)$. Moreover, in the proof, the condition $p \equiv 3 \mod{4}$ is also required in order to parameterize a rigid motion as a point in $\mathbb{F}_p^3$.

Based on these observations, the following is a natural extension of Theorem \ref{rational}.

\begin{conjecture}\label{rotation-conj}
Assume $p\equiv 3\mod 4$. Let $A, B\subset \mathbb{F}_p^d$, $d\ge 3$, with $|A|,~ |B|\ge p^{\frac{d}{2}}$. Then for almost every $g\in SO(d)$, one has $|\Delta(A, gB)|\gg p$.
\end{conjecture}

In higher dimensions, we are not sure if the condition $p\equiv 3\mod 4$ is needed. The reason is that, unlike in the plane, the stabilizer of an isotropic subspace in \(\mathbb{F}_p^d\), $d\ge 3$, is negligible compared to the size of \(SO(d)\).

    This conjecture is broader than the original Erd\H{o}s--Falconer distance problem, which is restricted to even dimensions.

In higher dimensions $\mathbb{F}_p^d$, one can parameterize an incidence between a point and a rigid motion as an incidence between a point and a $\binom{d}{2}$-plane in $\mathbb{F}_p^{\binom{d}{2}+d}$. This is a technical step, but it is crucial for adapting incidence geometry methods. Following the argument in two dimensions requires an extension of the work of Ellenberg and Hablicsek \cite{EHH} in $\mathbb{F}_p^d$. This may pose significant challenges due to algebraic structures.

In the continuous setting, since the celebrated Falconer distance conjecture remains wide open, we propose the following rotational analogue, which may be more tractable. For recent progress on the Falconer distance conjecture, we refer the reader to \cite{Du3, Du4, alex-fal, ShW} for more details. 

\begin{conjecture}
Let $A, B$ be compact sets in $\mathbb{R}^d$, $d\ge 2$, with $\dim_H(A),~ \dim_H(B)>\frac{d}{2}$. Then, for almost every $g\in SO(d)$, one has $\mathcal{L}^1(\Delta(A, gB))>0$.
\end{conjecture}

%\thang{Need to comment on the difference between the discrete and continuous settings....}
If $g$ is the identity matrix, although the exponent $\frac{5}{4}$ has been proved in the discrete and continuous settings in \cite{MPPRS} and \cite{alex-fal}, respectively, the two papers present a significant difference in methods and natures. More precisely, while the paper \cite{alex-fal} uses decoupling estimates and projection theorems, and the paper \cite{MPPRS} employs algebraic methods and incidence bounds. Moreover, it is known that when the Hausdorff dimension of a set is less than $\frac{4}{3}$, the $L^2$-distance norm is not bounded due to the train-track example in \cite{KTao}, but this is not true over prime fields, at least when $p\equiv 3\mod 4$.

Theorem \ref{rational} implies $|\Delta(A, gB)|\gg p$ for almost every $g\in O(2)$, but we do not know if the $L^2$-distance norm is bounded uniformly. In particular, for each $g\in O(2)$, define $\nu_g(t)$ as the number of pairs $(x, y)\in A\times B$ such that $||x-gy||=t$, then the question is on bounding the sum $\sum_{g\in O(2)}\sum_{t\in \mathbb{F}_p}\nu_g(t)^2$ from above. It would be interesting to see if it is at most $|A|^2|B|^2|O(2)|/p$ when $|A|, |B|\sim p$, or for almost every $g$ we have $\sum_{t\in \mathbb{F}_p}\nu_g(t)^2$ is at most $|A|^2|B|^2/p$ when $|A|, |B|\sim p$. Notice that for a fixed tuple $(x, y, x', y')\in A\times B\times A\times B$, there might exist many $g\in O(2)$ such that $||x-gy||=||x'-gy'||$.

Regarding the conjecture for the \textbf{intersection pattern I} in even dimensions, based on the connection to the Erd\H{o}s-Falconer distance problem, the sharpness examples in \cref{sub-4142}, and the fact that the condition $|A||B|\gg q^d$ cannot be replaced by $|A||B|\gg q^{d-\epsilon}$ for any $\epsilon>0$ (since $|A-gB|\le |A||B|$), we are led to the following conjecture. 

\begin{conjecture}
    Let $A$ and $B$ be sets in $\mathbb{F}_q^d$ with $d\ge 2$ even. Assume that $|A||B|\gg q^d$, then for almost every $g\in O(d)$, there are at least $\gg q^d$ elements $z$ satisfying $$|A\cap (g(B)+z)| \sim \frac{|A||B|}{q^d}.$$ 
\end{conjecture}

We briefly mention here the main difference between even and odd dimensions when studying this topic and related questions. Let $H$ be an isotropic subspace of maximal dimension in $\mathbb{F}_q^d$. It is well known that if $d$ is even, then $|H| = q^{d/2}$, whereas if $d$ is odd, then $|H| = q^{(d-1)/2}$. A construction and proof of this fact can be found in \cite{TAMS}. Roughly speaking, this discrepancy allows for different types of constructions in odd dimensions that are not available in even dimensions, and this distinction underlies several sharpness examples in related problems.

\subsection{Proof of Theorems \ref{thm0.3} -- \ref{rational}}

To prove Theorem \ref{thm0.3}, we recall the following result due to Ellenberg and Hablicsek \cite{EHH} on incidences between points and lines in $\mathbb{F}_p^3$, which settles a conjecture of Bourgain. 
\begin{theorem}\label{EH}
Let $\mathbb{F}_p$ be a prime field and let $L$ be a set of $N^{2}$ lines in $\mathbb{F}_p^{3}$ with $N\le p$, such that no $2N$ lines lie in any plane. Let $S$ be a set of points such that each line in $L$ contains at least $N$ points of $S$. Then $|S|>cN^{3}$ for some absolute constant $c>0$.
\end{theorem}

\begin{proof}[Proof of \textup{Theorem \ref{thm0.3}}]
Without loss of generality, we assume that $|A|=|B|=p$. It is well-known that there exists a transformation $T$ such that if $x=gy+z$, i.e. an incidence between $(x, y)$ and the rigid-motion $(g, z)$, then, under $T$, we have a point-line incidence bound in $\mathbb{F}_p^3$. Such a transformation can be found in \cite[Lemmas 2.1 and 2.2]{B}. In particular, each pair of points $(x, y)$ will be mapped to a line in $\mathbb{F}_p^3$, and each pair $(g, z)$ will be mapped to a point in $\mathbb{F}_p^3$. Let $L$ be the set of corresponding lines, and $S$ be the union of the lines in $L$. We have $|L|=p^2$, since the lines are distinct. Under the map $T$, it is not hard to check that any plane contains at most $2p$ lines from $L$, more details can be found in \cite[Section 7]{pham}. Thus, applying Theorem \ref{EH} implies $|S|\gg p^3$. We observe that if $x=gy+z$ for some $x\in A$ and $y\in B$ then $|A\cap (g(B)+z)|\ge 1$. Thus, there are at least $\gg p$ matrices $g\in SO(2)$ such that for each $g$ the number of elements $z\in \mathbb{F}_p^2$ with $|A\cap (g(B)+z)|\ge 1$ is at least $\gg p^2$. This completes the proof.
\end{proof}
\begin{proof}[Proof of \textup{Theorem \ref{rational}}]
It follows from Theorem \ref{thm0.3} that for almost every $g\in SO(2)$, there are at least $\gg p^2$ elements $z$ satisfying $$|A\cap (g(B)+z)| \ge 1.$$

For each such $g$, let $X_g=\{||z||\colon |A\cap (g(B)+z)| \ge 1\}$. We have $|X_g|\gg p$. Since $X_g\subset \Delta(A, gB)$, the theorem follows.
\end{proof}

\subsection{Sharpness of Theorem \ref{rational}}\label{sharp-cys}
\begin{proposition}\label{proposition7.2}
    Let $p = 2^n - 1$ be a Mersenne prime. For any $c\in (0, 1)$, there exist sets $A, B\subset \mathbb{F}_p^2$ with $|A|=|B|=M\sim p^{1-c}$ such that for all $g\in SO(2)$, one has $|\Delta(A, gB)|\le M$.
\end{proposition}
\begin{proof}
    Since $p$ is an odd prime with $p \equiv 3 \mod{4}$, we identify $\mathbb{F}_p^2$ with the quadratic extension $\mathbb{F}_{p^2}$ via the correspondence
$$(x_1, x_2) \longleftrightarrow z = x_1 + ix_2,$$
where $i^2 = -1 \in \mathbb{F}_{p^2} \setminus \mathbb{F}_p$. The condition $p \equiv 3 \mod{4}$ ensures that $-1$ is not a quadratic residue in $\mathbb{F}_p$, guaranteeing that $i \notin \mathbb{F}_p$.

Define the Frobenius conjugation $\bar{z} = z^p = x_1 - ix_2$, the field norm $N(z) = z\bar{z} \in \mathbb{F}_p$, and the real part $\mathtt{Re}(z) = \frac{z + \bar{z}}{2} \in \mathbb{F}_p$. Crucially, for $z = x_1 + ix_2$, we have
$$N(z) = x_1^2 + x_2^2,$$
establishing that the distance function in $\mathbb{F}_p^2$ corresponds precisely to the field norm in $\mathbb{F}_{p^2}$.

Let $U = \{u \in \mathbb{F}_{p^2} : N(u) = 1\}$ denote the group of unit norm elements, which forms a cyclic group of order $|U| = p + 1$. For $u \in U$, the multiplication map
$$R_u : \mathbb{F}_{p^2} \to \mathbb{F}_{p^2}, \quad z \mapsto uz$$
is $\mathbb{F}_p$-linear and norm-preserving: $N(uz) = N(u)N(z) = N(z)$. Hence, $R_u \in SO(2)$. Conversely, one can verify that every rotation in $SO(2)$ arises as $R_u$ for some $u \in U$.

{\bf Claim 1:} For $a, b, u \in \mathbb{F}_{p^2}$ with $N(a) = N(b) = N(u) = 1$,
$$N(a - ub) = 2 - 2\cdot\mathtt{Re}(a\bar{u}\bar{b}).$$
Here, $\mathtt{Re}(z)$ means the real part of $z$. 

Indeed, using the properties $\overline{xy} = \bar{x}\bar{y}$ and $a\bar{a} = b\bar{b} = u\bar{u} = 1$, we compute:
\begin{align*}
N(a - ub) &= (a - ub)\overline{(a - ub)} = (a - ub)(\bar{a} - \bar{u}\bar{b})\\
&= a\bar{a} - a\bar{u}\bar{b} - ub\bar{a} + ub\bar{u}\bar{b}\\
&= 2 - (a\bar{u}\bar{b} + \overline{a\bar{u}\bar{b}})\\
&= 2 - 2\cdot\mathtt{Re}(a\bar{u}\bar{b}). 
\end{align*}

Fix a divisor $M \mid (p + 1)$ and let $H \leq U$ be the unique subgroup of order $|H| = M$.

{\bf Claim 2:}
Let $A = B = H \subseteq \mathbb{F}_{p^2} \cong \mathbb{F}_p^2$. Then for every rotation $g = R_u$ with $u \in U$,
$$|\Delta(A, gB)| \leq M.$$

Indeed, for $a, b \in H$ and $u \in U$, Claim 1 yields
$$N(a - ub) = 2 - 2\cdot\mathtt{Re}(a\bar{u}\bar{b}).$$

Since $H \subseteq U$, we have $\bar{b} = b^{-1} \in H$. Furthermore, we have $\{ab^{-1} : a, b \in H\} = H$ by the group property. 
Thus,
$$\{a\bar{u}\bar{b} : a, b \in H\} = \bar{u} \cdot H,$$
which is a coset of $H$ in $U$ and hence has cardinality $M$. The map $z \mapsto 2 - 2\cdot\mathtt{Re}(z)$ sends $\bar{u} \cdot H$ into $\mathbb{F}_p$, so its image has size at most $M$. 

Choosing $M \mid (p + 1)$ with $M \sim p^{1-c}$ and for any $c > 0$, and let $H \leq U$ be the subgroup of order $M$. Taking $A = B = H$, Claim 2 immediately gives $|\Delta(A, gB)| \leq M$ for all rotations $g$. 
\end{proof}
\section{Intersection pattern II}\label{subsec: 1}
For $g\in O(d)$, define the map $S_g\colon \mathbb{F}_q^{2d}\to \mathbb{F}_q^d$ by 
\[S_g(x, y)=x-gy.\]
The results in the previous section imply that if $P=A\times B$ with $A, B\subset \mathbb{F}_q^d$ and $|A||B|$ is larger than a certain threshold, then for almost every $g\in O(d)$, one has $|S_g(P)|\gg q^d$. 

In this section, we present a result for general sets $P\subset \mathbb{F}_q^{2d}$.  With the same approach, we have the following theorem.
\begin{theorem}\label{Od}
    Given $P\subset \mathbb{F}_q^{2d}$, there exists $E\subset O(d)$ with $|E|\ll q^{2d}|O(d-1)|/|P|$ such that, for all $g\in O(d)\setminus E$, we have $|S_g(P)|\gg q^d$. 
\end{theorem}

 We always have $|S_g(P)|q^d\ge |P|$, since for each $z\in S_g(P)$ and for each $x\in \mathbb{F}_q^d$, there is at most one $y\in \mathbb{F}_q^d$ such that $(x, y)\in P$. Thus, $|S_g(P)|\ge |P|q^{-d}$ for all $g\in O(d)$.

In the theorem above, to ensure that \( |E| < |O(d)| \), the condition \( |P|\gg q^{d+1} \) is required. If \( P=A\times B \) with \( A,B\subset \mathbb{F}_q^d \), then the conclusion \( |S_g(P)|\gg q^d \) means that there exist at least \(\gg q^d\) elements \(z\) such that  
\begin{equation}\label{eq11}
|A\cap (gB+z)|\ge 1.
\end{equation}
Theorem \ref{thm0.3} provides an optimal improvement of this statement in two dimensions over prime fields. It is therefore natural to ask about the optimal condition in higher dimensions and over general finite fields. To formulate the correct conjecture, one must determine the largest exponent \(\alpha\) such that for any \(\epsilon>0\), there exists a set \(P\) of size about \(q^{\alpha-\epsilon}\) for which  
\[
|S_g(P)|\le q^{d-\epsilon'}
\]  
for some \(\epsilon'=\epsilon'(\epsilon)>0\), and for at least \((1-o(1))|O(d)|\) matrices \(g\in O(d)\). In Proposition \ref{prop8.1}, we construct such an example with \(\alpha\ge d\). Thus, the right exponent $\alpha$ belongs to the range $(d, d+1)$. It is worth noting that by setting $P=A\times B$, it follows from (\ref{eq11}) that for almost every $g\in SO(d)$, one has $|\Delta(A, gB)|\gg p$, provided that $|A||B|\gg q^{\alpha}$. If $\alpha=\frac{d}{2}$, this would match Conjecture \ref{rotation-conj}.
\subsection{Proof of Theorem \ref{Od}}
Theorem \ref{thmNov17} follows immediately from Theorems \ref{Od} and \ref{thm0.3}.
\begin{proof}[Proof of \textup{\cref{Od}}]
Let $E$ be the set of $g$ in $O(d)$ such that $|S_{g}(P)| < q^{d} / 2$ and $R=\{(g, S_g(P))\colon g\in E\}$. 
By \cref{thm: incidence inequality}, we first observe that 
\[I(P, R)\le \frac{|P||R|}{q^{d}}+|O(d-1)|^{1/2}q^{d/2}|P|^{1/2}|R|^{1/2}.\]
Note that $|R| < |E|q^{d}/2$ and $I(P, R)=|P||E|$.
This infers 
\[|P||E| \le q^{d}|O(d-1)|^{1/2}|P|^{1/2}|E|^{1/2}.\]
So,
\[|E|\ll \frac{q^{2d}|O(d-1)|}{|P|},\]
as desired.
\end{proof}
\subsection{Sharpness of Theorem \ref{Od}}

\begin{proposition}\label{prop8.1}
For every $0<\epsilon\le 1$ there exists a set
$P\subset \mathbb{F}_q^{2d}$ with $|P|\sim q^{\,d-\epsilon}$ such that for every $g\in O(d)$, one has
$|S_g(P)|\ \le\ q^{\,d-\epsilon}.$

\end{proposition}

\begin{proof}
Fix a nonzero linear functional $\ell:\mathbb{F}_q^d\to\mathbb{F}_q$ (the projection onto the first coordinate), and let $T\subset \mathbb{F}_q$ and $Y\subset \mathbb{F}_q^d$ such that $|T|=|Y|\sim q^{\frac{1-\epsilon}{2}}$.

Define $A:=\{x\in\mathbb{F}_q^d:\ \ell(x)\in T\}$. Then $|A|=|T|\,q^{d-1}\sim q^{\,d-(1+\epsilon)/2}$. Let $P:=A\times Y$. Thus,
\[
|P|=|A|\,|Y|\sim q^{\,d-(1+\epsilon)/2}\cdot q^{(1-\epsilon)/2}
= q^{\,d-\epsilon}.
\]

For any $g\in O(d)$ and $(x,y)\in A\times Y$, we have
\[
\ell(x-gy)=\ell(x)-(\ell\circ g)(y)\in T-(\ell\circ g)(Y),
\]
hence, $S_g(P)\subset \ell^{-1}\!\big(T-(\ell\circ g)(Y)\big)$. Since $|\ell^{-1}(U)|=|U|\,q^{d-1}$, for any $U\subset\mathbb{F}_q$, so
\[
|S_g(P)|\le |\ell^{-1}(T-(\ell\circ g)(Y))|
= |T-(\ell\circ g)(Y)|\,q^{d-1}.
\]
Moreover,
\[
|T-(\ell\circ g)(Y)|\le |T||Y| \le\ q^{\,1-\epsilon}.
\]
Combining all estimates gives $|S_g(P)|\le q^{\,1-\epsilon}\,q^{d-1}=q^{\,d-\epsilon}$, uniformly in $g$. This completes the proof.
\end{proof}

\section{Growth estimates under orthogonal matrices}\label{growth-9}
Let $A,B\subset \mathbb{F}_q^d$. As shown in previous sections, there exists an exceptional set $E\subset O(d)$ such that for all $g\in O(d)\setminus E$,
\[
|A-gB|\gg q^{d}.
\]
In this section, we reformulate this in the language of expanding functions. Assuming $|A|\le |B|$ and fixing $\epsilon>0$, we aim for a weaker, scale-sensitive bound
\[
|A-gB|\gg |B|^{1+\epsilon}
\]
for all $g\in O(d)\setminus E_{\epsilon}$. With a unified proof, we have the following theorems. Note that since $\epsilon$ is arbitrary, constructing matching examples across the full range is delicate and appears difficult.
\begin{theorem}\label{thm: growth q}
Let $\epsilon>0$. Given $A,B \subset \mathbb{F}_q^{d}$ with $|A| \le |B|$ and $|B|^{1+\epsilon}<q^{d}/2$, there exists $E\subset O(d)$ with 
\[|E| \ll  \frac{|O(d-1)|q^d|B|^{\epsilon}}{|A|} \] 
such that for all $g\in O(d)\setminus E$, we have 
\[|A-gB| \gg |B|^{1+\epsilon}.\]
\end{theorem}
\begin{theorem}\label{thm: growth dq}
Let $\epsilon>0$. Assume either ($d\ge 3$ odd) or ($d\equiv 2\mod 4, q\equiv 3\mod 4$). Given $A,B \subset \mathbb{F}_q^{d}$ with $|A| \le |B|$ and $|B|^{1+\epsilon}<q^{d}/2$, there exists $E\subset O(d)$
such that for all $g\in O(d)\setminus E$, we have 
\[|A-gB| \gg |B|^{1+\epsilon}.\]
In particular, 
\begin{enumerate}
    \item[\textup{(1)}] If $|A|<q^{\frac{d-1}{2}}$, then one has $|E|\ll \frac{q^{\frac{d^{2}-d}{2}}|B|^{\epsilon}}{|A|}$.
    \item[\textup{(2)}] If $q^{\frac{d-1}{2}}\le |A|\le q^{\frac{d+1}{2}}$, then one has
$|E|\ll q^{\frac{d^{2}-2d+1}{2}}|B|^{\epsilon}$.
\end{enumerate}
\end{theorem}
Under $|A|\le |B|$ and $|B|^{1+\epsilon}<q^d/2$, the following table presents conditions on $A$ and $B$ such that $|E|<|O(d)|$.

\vspace{2mm}
\begin{table}[!ht]
    \centering
   	\begin{tabular}{ |c|c|c|c|c} 
 \hline
  & Theorem \ref{thm: growth q} & Theorem \ref{thm: growth dq} (1)& Theorem \ref{thm: growth dq} (2)\\ 
 \hline
 $|E|<|O(d)|$ & $|A|\gg q|B|^{\epsilon}$ &$|A|<q^{\frac{d-1}{2}}$, ~ $|A|\gg |B|^\epsilon$ & $q^{\frac{d-1}{2}}\le |A|\le q^{\frac{d+1}{2}}$, ~$|B|\ll q^{\frac{d-1}{2\epsilon}} $\\
 \hline 
\end{tabular}
\vspace{5mm}
\end{table}
Hence, assume either ($d\ge 3$ odd) or ($d\equiv 2\mod 4, q\equiv 3\mod 4$), and $|A|=|B|=N\le q^{\frac{d-1}{2}}$, then $|A-gB|\gg N^2$ for all $g\in O(d)\setminus E$.

In two dimensions, the statement can be stated in a much cleaner way:  if $|A|=q^x$ and $|B|=q^y$, as long as $0<x\le y<2$, then for almost every $g\in O(2)$, we can always find $\epsilon=\epsilon(x, y)>0$, close to $\frac{x}{2y}$, such that $|A-gB|\gg |B|^{1+\epsilon}$.
\begin{theorem}\label{thm: growth d=2q}
Let $\epsilon>0$. Assume that $q \equiv 3 \mod{4}$. Given $A,B \subset \mathbb{F}_q^{2}$ with $|A| \le |B|$ and $|B|^{1+\epsilon}<q^{2}/2$, there exists $E\subset O(2)$ with 
\[|E|\ll \frac{q|B|^{\epsilon}}{|A|^{1/2}}\]
such that for all $g\in O(2)\setminus E$, we have 
\[|A-gB| \gg |B|^{1+\epsilon}.\]
\end{theorem}
The conditions on the sizes of $A$ and $B$, such that the exceptional set $E$ is of size smaller than $|O(2)|$, can be improved further  when we replace $\mathbb{F}_q$ by $\mathbb{F}_p$. 
\begin{theorem}[\textbf{Small $A$}]\label{thm: growth d2p1}
Let $\epsilon>0$. Assume that $p \equiv 3 \mod{4}$. Given $A,B \subset \mathbb{F}_p^{2}$ with $|A| \le |B|$ and $|B|^{1+\epsilon}<p^{2}/2$, there exists $E\subset O(2)$ such that for all $g\in O(2)\setminus E$, we have 
\[|A-gB| \gg |B|^{1+\epsilon}.\]
In particular,
\begin{enumerate}

\item[\textup{(1)}] If $p^{3/4} \le |A| \le p $ and $p^{5/4}\le |B|\le p^{4/3}$, then 
$|E|\ll \frac{p^{1/8}|B|^{1/2 + \epsilon}}{|A|^{1/3}}$.

\item[\textup{(2)}] If $|A|\le p$ and $p\le |B|\le p^{5/4}$, then
$|E|\ll \frac{p|B|^{\epsilon} + p^{1/3}|A|^{2/3}|B|^{1/3+\epsilon} }{|A|}$.

\item[\textup{(3)}] If $|A|\le p$ and $|B|\le p$, then
$|E|\ll \frac{p|B|^{\epsilon}+|A|^{2/3}|B|^{2/3+\epsilon}}{|A|}$.
\end{enumerate}
\end{theorem}

\begin{theorem}[\textbf{Medium $A$}]\label{thm: growth d2p2}
Let $\epsilon>0$. Assume that $p \equiv 3 \mod{4}$. Given $A,B \subset \mathbb{F}_p^{2}$ with $|A| \le |B|$ and $|B|^{1+\epsilon}<p^{2}/2$, there exists $E\subset O(2)$ such that for all $g\in O(2)\setminus E$, we have 
\[|A-gB| \gg |B|^{1+\epsilon}.\]
In particular,
\begin{enumerate}
\item[\textup{(1)}] If $p \le |A| \le p^{5/4}$ and $p \le |B| \le p^{5/4}$, then
$|E|\ll \frac{p^{2/3}|B|^{1/3+\epsilon}}{|A|^{2/3}}$.

\item[\textup{(2)}] If $p \le |A| \le p^{5/4}$ and $p^{5/4}\le |B|\le p^{4/3}$, then 
$|E|\ll \frac{p^{11/24}|B|^{1/2 + \epsilon}}{|A|^{2/3}}$.

\item[\textup{(3)}] If $p^{5/4}\le |A|\le p^{4/3}$ and $p^{5/4}\le |B|\le p^{4/3}$, then $|E|\ll \frac{p^{1/4}|B|^{1/2 + \epsilon}}{|A|^{1/2}}$.
\end{enumerate}
\end{theorem}

\begin{theorem}[\textbf{Large $B$}]\label{thm: growth d2p3}
Let $\epsilon>0$. Assume that $p \equiv 3 \mod{4}$. Given $A,B \subset \mathbb{F}_p^{2}$ with $|A| \le |B|$ and $|B|^{1+\epsilon}<p^{2}/2$, there exists $E\subset O(2)$ such that for all $g\in O(2)\setminus E$, we have 
\[|A-gB| \gg |B|^{1+\epsilon}.\]
In particular,
\begin{enumerate}

\item[\textup{(1)}] If $p\le |A|\le p^{5/4}$ and $|B|>p^{4/3}$, then 
$|E|\ll \frac{p^{5/6}|B|^{1/4 +\epsilon}}{|A|^{2/3}}$.

\item[\textup{(2)}] If $p^{5/4}\le |A|\le p^{4/3}$ and $|B|>p^{4/3}$, then
$|E|\ll \frac{p^{5/8}|B|^{1/4 + \epsilon}}{|A|^{1/2}}$.
\end{enumerate}
\end{theorem}

Compared to Theorem \ref{thm: growth d=2q}, conditions for improvements over prime fields are summarized in the following table.

\vspace{5mm}
\begin{table}[!ht]
    \centering
   	\begin{tabular}{ |c|c|c|c|c|c|c|} 
 \hline
  & $|A|\le p^{\frac{1}{2}}$ & $p^{\frac{1}{2}}<|A|\le p^{\frac{3}{4}}$&$p^{\frac{3}{4}}< |A|\le p$& $p<|A|\le p^{\frac{5}{4}}$&$p^{\frac{5}{4}}< |A|\le p^{\frac{4}{3}}$ \\ 
 \hline
 $|B|\le p^{\frac{3}{4}}$ & $\surd$ & $\surd$ & $\surd$ &$\diagup $ &$\diagup $\\
 \hline 
 $p^{\frac{3}{4}}<|B|\le p$ & $\surd$ & $\surd$ & $\surd$& $\diagup $ &$\diagup $\\
 \hline
 $p<|B|\le p^{\frac{5}{4}}$ & $\surd$ & $\surd$& $\surd$& $\surd$&$\diagup $\\
 \hline
 $p^{\frac{5}{4}}<|B|\le p^{\frac{4}{3}}$ & $\varnothing$ & $\varnothing$& $\surd$&$\surd$&$\surd$\\
 \hline
 $p^{\frac{4}{3}}<|B|$ & $\varnothing$ & $\varnothing$ & $\varnothing$&$|B|^3<p^2|A|^2$&$|B|<p^{\frac{3}{2}}$\\
 \hline 
\end{tabular}
\vspace{5mm}
    \caption{In this table, by $``\surd"$ we mean better result, by $``\varnothing"$ we mean weaker result, by $``f(|A|, |B|)"$ we mean better result under the condition $f(|A|, |B|)$, and by $``\diagup"$ we mean invalid range corresponding to $|B|\le |A|$.}
\end{table}

Compared to Theorem \ref{thm: growth dq} with $d=2$, conditions for improvements over prime fields are summarized in the following table.

\vspace{5mm}
\begin{table}[!ht]
    \centering
   	\begin{tabular}{ |c|c|c|c|c|c|c|} 
 \hline
  & $|A|\le p^{\frac{1}{2}}$ & $p^{\frac{1}{2}}<|A|\le p^{\frac{3}{4}}$&$p^{\frac{3}{4}}< |A|\le p$& $p<|A|\le p^{\frac{5}{4}}$&$p^{\frac{5}{4}}< |A|\le p^{\frac{4}{3}}$ \\ 
 \hline
 $|B|\le p^{\frac{3}{4}}$ & $\varnothing$ & $\surd$ & $\surd$ &$\diagup $ &$\diagup $\\
 \hline 
 $p^{\frac{3}{4}}<|B|\le p$ & $\varnothing$ & $\surd$ & $\surd$& $\diagup $ &$\diagup $\\
 \hline
 $p<|B|\le p^{\frac{5}{4}}$ & $\varnothing$ & $|B|^2\le p|A|^2$& $|B|^2\le p|A|^2$& $\surd$&$\diagup $\\
 \hline
 $p^{\frac{5}{4}}<|B|\le p^{\frac{4}{3}}$ & $\varnothing$ & $\varnothing$& $|B|\le p^{3/4}|A|^{2/3}$&$\surd$&$\surd$\\
 \hline
 $p^{\frac{4}{3}}<|B|$ & $\varnothing$ & $\varnothing$ & $\varnothing$&$p^4|B|^3\le |A|^8$&$\surd$\\
 \hline 
\end{tabular}
\vspace{5mm}
    \caption{In this table, by $``\surd"$ we mean better result, by $``\varnothing"$ we mean weaker result, by $``f(|A|, |B|)"$ we mean better result under the condition $f(|A|, |B|)$, and by $``\diagup"$ we mean invalid range corresponding to $|B|\le |A|$.}
\end{table}

%\thang{Need to run a random check}

\subsection{Proof of Theorems \ref{thm: growth q} -- \ref{thm: growth d2p3}}

\begin{proof}[Proof of \textup{\cref{thm: growth q}}]
Set $P=A\times B$. Let $\lambda=|B|^{1+\epsilon}$. Define 
\[E=\{g\in O(d)\colon |A-gB|\le \lambda\}.\]
Set $R=\{(g, z)\colon z\in A-gB, g\in E\}$. We observe that $I(P, R)=|P||E|$. 
Applying \cref{thm: incidence inequality}, one has a constant $C>0$ such that
\[|P||E|=I(P, R)\le \frac{|P||R|}{q^d}+C\sqrt{|O(d-1)|}q^{d/2}|P|^{1/2}|R|^{1/2}.\]
Using the fact that $|R|\le \lambda|E|$ with $\lambda <q^d/2$, we have 
\[\frac{|P||E|}{2}\le C\sqrt{|O(d-1)|}q^{d/2}|P|^{1/2}|R|^{1/2}.\]
This implies 
\[|E|\ll \frac{|O(d-1)|\lambda q^d}{|A||B|}.\]
This completes the proof.
\end{proof}

\begin{proof}[Proof of \textup{\cref{thm: growth dq}}]
We use the same notations as in \cref{thm: growth q}, and assume that $|A|<q^{\frac{d-1}{2}}$, since the other case can be proved similarly. 
By \cref{thm: incidence1}, one has a constant $C>0$ such that
\[|P||E|=I(P, R) \le \frac{|P||R|}{q^d}+Cq^{(d^2-d)/4}|P|^{1/2}|R|^{1/2}.\]
Using the fact that $|R|\le \lambda|E|$ with $\lambda <q^d/2$, we have 
\[\frac{\sqrt{|P||E|}}{2}\le C\sqrt{\lambda}q^{(d^2-d)/4}.\]
This implies 
\[|E|\ll \frac{\lambda q^{(d^2-d)/2}}{|A||B|},\]
as desired.
\end{proof}
\begin{proof}[Proof of \textup{\cref{thm: growth d=2q}}]
We use the same notations as in \cref{thm: growth q} for $d=2$.
Applying \cref{thm: incidences2}, there exists a constant $C>0$ such that 
\[|P||E|=I(P, R)\le \frac{|P||R|}{q^2}+Cq^{1/2}|P|^{1/2}|R|^{1/2}|A|^{1/4}.\]
Using the fact that $|R|\le \lambda|E|$ with $\lambda <q^d/2$, we have 
\[\frac{\sqrt{|P||E|}}{2}\le C q^{1/2}\lambda^{1/2}|A|^{1/4}.\]
This implies
\[|E|\ll \frac{\lambda q}{|A|^{1/2}|B|},\]
as desired.
\end{proof}
\cref{thm: growth d2p1} (1) and (2), \cref{thm: growth d2p2}, and \cref{thm: growth d2p3} are proved by the same approach using \cref{thm: incidences31}, \cref{thm: incidences3}, and \cref{thm: incidences32}, respectively.

For \cref{thm: growth d2p1} (3), the same proof implies that if $|A|\le p$ and $|B|\le p$, then
$|E|\ll \frac{p|B|^{\epsilon}+|B|^{\epsilon}|P|^{2/3}}{|A|}$.

With the approach as in the proof of Theorem~\ref{thm0.3}, if we use Theorem \ref{EH} on incidences between points and lines in $\mathbb{F}_p^3$ due to Ellenberg and Hablicsek in~\cite{EHH} in place of incidence results between points and rigid motions, then, compared to Corollary \ref{cosep7}, a weaker estimate can be derived. More precisely, for almost every $g \in O(2)$, we have
\[
|A - gB| \gg \frac{N^3}{p}, ~|A|=|B|=N\le p.
\]

\section{Acknowledgements}
T. Pham would like to thank the Vietnam Institute for Advanced Study in Mathematics (VIASM) for the hospitality and for the excellent working condition. S. Yoo was supported by the Korea Institute for Advanced Study (CG082701) and by the Institute for Basic Science (IBS-R029-C1).

\bibliographystyle{amsplain}

\end{document}